\documentclass[12pt,a4paper]{amsart}    %\documentclass[12pt]{article}   %amsart

\usepackage{amsfonts, amssymb, amsmath, amsthm}
\usepackage{mathrsfs}

\usepackage{appendix}
\usepackage{float}
\usepackage{fancyhdr}  
\usepackage{caption}

\usepackage{tikz}
\usepackage[all]{xy}

\usetikzlibrary{arrows}
\usepackage{hyperref}
\usepackage{color,hyperref}
    \catcode`\_=11\relax

    \newcommand\emailampersat{{\color{red}\small@}}
    \catcode`\_=8\relax

\theoremstyle{plain}
\newtheorem{thm}{Theorem}[section]
\newtheorem{lemma}[thm]{Lemma}
\newtheorem{cor}[thm]{Corollary}
\newtheorem{remark}[thm]{Remark}
\newtheorem{example}[thm]{Example}

\theoremstyle{definition}

\newcommand{\be}{\begin{equation}} % begin equation
\newcommand{\ee}{\end{equation}}

\newcommand{\sa}{\Sigma}
\newcommand{\M}{\mathbf{Meas}}
\newcommand{\Set}{\mathbf{Set}}
\newcommand{\Cvx}{\mathbf{Cvx}}

\newcommand{\C}{\mathcal{C}}
\newcommand{\K}{\mathcal{K}}
\newcommand{\E}{\downarrow}
\newcommand{\leb}{\mathcal{L}}
\newcommand{\two}{\mathbf{2}}
\newcommand{\one}{\mathbf{1}}
\newcommand{\G}{\mathcal{G}}
\newcommand{\T}{\mathcal{P}}
\newcommand{\I}{\mathbb{I}}
\newcommand{\mbfI}{\mathbf{I}}

\newcommand{\gen}[1]{\langle \langle #1 \rangle \rangle}

\begin{document}
%\frontmatter
\title[Giry algebras and Convex Spaces]{The equivalence between the categories of Giry algebras and convex spaces}
\author{Kirk Sturtz}
\address{Kirk Sturtz, Nong Saeng, Thailand.}
\email{kirksturtz@universalmath.com}


\begin{abstract}    A duality between the category of convex spaces and measurable spaces arises from the existence of the unit interval, which is an object in both these categories.  The full subcategory of the category of convex spaces, consisting of just the single object, the unit interval,  is both a dense and codense subcategory in the category of convex spaces.  Combined with the the symmetric monoidal closed category structure of the category, one obtains the double dualization monad into the unit interval, which sends a point to the evaluation map at that point.  The restriction of the codomain of the unit of this monad  to the weakly averaging affine functionals is an isomorphism.  Moreover, every convex space has an associated  measurable space,  whose  $\sigma$-algebra is generated by the Boolean subobjects of that convex space.  The resulting $\sigma$-algebra of that measurable space makes it a separated measurable space. These properties are used to give a proof that the category of Giry algebras is equivalent to the category of convex spaces.   
%The factorization of the Giry monad through the category of convex spaces arises, in part, by taking  the Boolean subobjects of a convex spaces as a generating set for the $\sigma$-algebra of a convex space.  
 \end{abstract}
  \maketitle
  
  \vspace{.1in}
  
Editorial Note: This paper has been updated and replaced with the article \emph{The factorization of the Giry monad and convex spaces as an extension of the Kleisi category}.  The approach used in this paper was originally  based upon using the \emph{density} of the unit interval in $\Cvx$, whereas the updated perspective uses the codensity of the unit interval in $\Cvx$.  The latter seems more natural and intuitive, particularly since the use of the full subcategory of separated measurable spaces provides motivation for the various constructions.  Moreover, since the analysis of the close relationship between the tensor products in the two categories, $\M$ and $\Cvx$, are developed (and of considerable interest in their own right), we chose to develop the theory from that perspective rather than correct and update a few annoying errors in this presentation. (The main error at issue is pointed out subsequently).  Nevertheless, the perspective of paths in convex spaces, provides a nice tool for analyzing several aspects associated with the theory.      
 
 \vspace{.1in}
 
\section{Introduction}  
The Giry monad $\G$ on the category of measurable spaces $\M$, defined for every space $X$ as the space of probability measures on $X$, has a 
natural convex structure associated with it, and conversely, the category of convex spaces, $\Cvx$, has for each object $A$ a natural  $\sigma$-algebra structure generated by the Boolean subobjects of $A$.

By using the symmetric monoidal closed category (SMCC) structure of both of these categories,  we show these two natural structures determine an adjoint pair of functors which factorize the Giry monad, and prove the category $\Cvx$ is equivalent to the category of Giry algebras, $\M^{\G}$ .\footnote{The Giry monad is named after Giry\cite{Giry}, however the original construction  follows from Lawvere's original paper in 1963, before ``monads'' were defined and their relationship to adjunctions were clarified.}  

In the literature the Giry monad is almost always accompanied by characterizations using the simplicial category 
$\Delta$ and the functor
\be \nonumber
\Delta \rightarrow \mathbf{Top}
\ee
mapping a finite ordinal $\mathbf{n}$ to the standard $n$-dimensional affine simplex $\Delta_n$, viewed as a subspace of  the Euclidean space $\mathbb{R}^{n+1}$.    While such explanations provide a useful perspective about the Giry monad, the assumption of an underlying topological structure is unnecessary and conceals  the underlying connection between convex and measurable spaces  necessary to prove $\Cvx$ is equivalent to $\M^{\G}$.   Approaches other than the simplicial method can be found in \cite{Doberkat, Keimel, Swirszcz}.
 
The category  $\Cvx$  has three basic properties,  
 (i) it is a symmetric monoidal closed category (SMCC) under the tensor product,  (ii) it is complete and cocomplete, and (iii) the full subcategory $\I \hookrightarrow \Cvx$ consisting of the unit interval $\mbfI = [0,1]$ is dense in $\Cvx$. The proof of the first two facts are both straightforward verifications. Details can be found in Meng\cite{Meng}. 

The category $\M$ is also complete, cocomplete, and a SMCC under the tensor product.  The monoidal structure, $(\M, \otimes, \mathbf{1})$, endows a pair of measurable spaces with the $\sigma$-algebra generated by all the constant graph maps, $\{ X \stackrel{\Gamma_y}{\longrightarrow} X \times Y \}_{x \in X}$ 
and $\{ Y \stackrel{\Gamma_x}{\longrightarrow} X \times Y \}_{y \in Y}$. This tensor product structure on the cartesian product is denoted $X \otimes Y$, and contains the product $\sigma$-algebra as a sub $\sigma$-algebra.\footnote{Further details on the tensor product structure on $\M$ making $\M$ a SMCC structure can be found in \cite{Sturtz}. The results in this paper were essentially given in \cite{Sturtz}, except we were now aware that  the category $\I$ was is dense and codense in $\Cvx$.  With that knowledge in hand, the assumption that each affine functional $\mbfI^A \stackrel{P}{\longrightarrow} \mbfI$ preserves limits of sequences of simple functions can be dropped, since the adequacy of $\I$ implies that assumption.

The full implications of this SMCC structure have yet to be worked out, and one important aspect of this SMCC structure, relating to probability theory, is that while  the Giry monad is a commutative monad with respect to the standard product structure,  with respect to the tensor product structure on $\M$ (which is necessary if one wants to work with function spaces), the question of the commutativity/noncommutativity of the Giry monad is unknown, although we conjecture it is noncommutative. (The commutativity of a monad is defined with respect to the monoidal structure.\cite{Kock})  While the calculus of extensive quantities (probability measures are extensive quantities) has been studied by A. Kock\cite{Kock2}, the applicability to $\M$ with the SMCC structure remains an open question.}
This \emph{tensor} product structure, rather than the standard product structure generated by the two coordinage projection mappings, is necessary to obtain the SMCC structure making the evaluation maps measurable functions.  
%We use this property - the evaluation maps are measurable -  to provide a proof of the  adjunction in question.   

An overview of the problem, showing $\Cvx$ is equivalent to $\M^{\G}$,  and an outline of the paper is as follows.
To  prove the Eilenberg-Moore category of $\G$-algebras is equivalent to the category of convex spaces,  we factor the Giry monad into two functors with the functor $\T$ at component $X \in_{ob} \M$
\begin{figure}[H]
\begin{equation}   \nonumber
 \begin{tikzpicture}[baseline=(current bounding box.center)]
         \node   (MG)  at   (0, 2)  {$\M^{\G}$};
         \node  (C)  at (0,0)    {$\M$};
         \node  (C2) at  (4,0)   {$\Cvx$};
         \node  (c)   at    (6.5, 0)   {$\T \dashv \mathbf{\Sigma}$};
         
         \draw[->, left] ([xshift=-2pt] C.north) to node {$\mathcal{F}^{\G}$} ([xshift=-2pt] MG.south);
         \draw[->, right] ([xshift=2pt] MG.south) to node {$\mathcal{U}^{\G}$} ([xshift=2pt] C.north);
        \draw[->,above] ([yshift=2pt] C.east) to node {$\T$} ([yshift=2pt] C2.west);
        \draw[->,below,dashed] ([yshift=-2pt] C2.west) to node {$\mathbf{\Sigma}$} ([yshift=-2pt] C.east);
        \draw[->, above] (C2) to node {$\Phi$} (MG);
	 \end{tikzpicture}
 \end{equation}
 \caption{The Eilenberg-Moore adjunction, \mbox{$\mathcal{F}^{\G} \dashv \mathcal{U}^{\G}$}, and the proposed adjunction to $\Cvx$, $\T \dashv \Sigma$.}
 \label{comparisonFunctor}
 \end{figure}
 \noindent
being given as $\G(X)$ viewed simply as a convex space, having no $\sigma$-algebra associated with it.    In this diagram we have employed the standard notation, with $\Phi$ being the comparison functor, and $\T \dashv \Sigma$ the (desired) adjunction such that the comparison functor is an equivalence of categories.
%\footnote{ Since  $\Cvx$ has colimits, and in particular coequalizers, it follows there exist a  left  adjoint $\hat{\Phi}$ to $\Phi$.  When this adjunction,  $\hat{\Phi} \dashv \Phi$,  which always exist when the category in question has coequalizers, has both the unit and counit  natural transformation  naturally isomorphic to the identity functors on $\M^{\G}$ and $\Cvx$, respectively, then the categories are equivalent.  It is this left adjunction to the comparison functor which leads to the various ``Beck conditions'' for tripleability, and the dissertation of Beck is, in our view, required reading to understand algebras, from  both the historical and technical view. }

 The functor $\T$ preserves colimits\footnote{The fact $\T$ preserves colimits is an elementary verification; the coproduct of $X$ and $Y$  in $\M$ is given by the disjoint union, $X+Y$, with the final $\sigma$-algebra on $X \cup Y$ such that the two inclusion maps $X \hookrightarrow X+Y$ and $Y \hookrightarrow X + Y$ are measurable.}, and in particular,\footnote{The symbol $\two$ is overloaded, as we use it to represent both a discrete measurable space, and a discrete convex space.  The distinction should be clear from the context.  Similiarly with regards to the unit interval $\mbfI$, and the object $\one$ which is the terminal object in both $\M$ and $\Cvx$. (Both can be viewed as measurable spaces with a natural convex structure.) } 
\be \nonumber
\T(\two) = \T(\one + \one) = \T(\one) + \T(\one) =\one+\one =  [0,1]
\ee
where the coproduct $\one +\one$ in $\Cvx$ is the unit interval, $\mbfI = [0,1]$, which has the  convex structure given by the free functor on $2$ generators.\footnote{In $\Cvx$, a coproduct of any two spaces $A$ and $B$ takes the set coproduct of $A$ and $B$, and takes the free convex space generated by those elements and then uses a equivalence relation on that set such that $
\alpha (a_1,1) + (1-\alpha) (a_2,1)  \cong (\alpha a_1  + (1-\alpha) a_2,1)$
and similarly for elements of $B$.
Thus the coproduct
$ A + B = \{ \displaystyle{ \sum_{i=1}^n} (1-\alpha_i) (a_i, 1) + \alpha_i (b_i,2) \, | \, a_i \in A, \, b_i \in B, \, \forall \alpha \in [0,1],  \, n \in \mathbb{N} \}$
The fact that $\Cvx$ is complete and cocomplete is well known.  In fact $\Cvx$ has a coseparator $\mathbb{R}_{\infty}$, as well as a separator (the one point space).\cite{Borger}}  
We identify  $\mbfI$ with $\T(\two)$, as convex spaces, using the isomorhism $\alpha \mapsto (1-\alpha) \delta_0 + \alpha \delta_1$,   where $\delta_0$ and $\delta_1$ are the dirac measures on the discrete measurable space $\two$.   

To prove the equivalence between $\M^{\G}$ and $\Cvx$, amounts to showing that there exists a right adjoint to the functor $\T$.  We can characterize this right adjoint $\Sigma$ as being defined by taking as the generating set for the $\sigma$-algebra on any convex space $A$, the set $\Cvx(A, \two)$.  Every element in the function space $\Cvx(A, \two)$  determines a \emph{Boolean subobject} pair of $A$ and is the topic of \S \ref{BooleanSubobjects}.  The set of all Boolean subobjects of a convex space generate the $\sigma$-algebra for $\Sigma A$, and the space $\Sigma A$ is a \emph{separated} measurable space, which is shown in \S \ref{separated} using the fact that $\Cvx$ has a coseparator.  The density and codensity of $\I$ are shown in \S \ref{density} and \S \ref{codensity}, respectively.  The codensity of $\I$ implies that the image of the double dualization map $A \rightarrow \Cvx( \Cvx(A, \mbfI), \mbfI)$ into $\mbfI$, sending $a \mapsto ev_a$, is an isomorphism, with the image space being the space of all weakly averaging affine functionals. This isomorphism, arising from the double dualization map into $\mbfI$ with the restricted codomain of all weakly averaging affine functionals, is, in many respects, the key point in proving the equivalence between $\Cvx$ and $\M^{\G}$.  
 \S \ref{altView} provides an alternative view of the codensity of $\I$ in $\Cvx$, while \S \ref{functions} provides a bridge from viewing function spaces, both $\mbfI^A$ or $\two^A$,  as convex spaces to viewing them as measurable spaces.  Both \S \ref{altView} and \S \ref{functions}  are provided to lend further understanding to duality and function spaces, respectively. They are not necessary to prove the main results, and may be skipped, referring back to them for clarification as needed.
  In \S \ref{adjunction} the two functors $\T$ and $\Sigma$ are shown to form an adjunct pair whose composite is the Giry monad, and then in \S \ref{equiv} the equivalence between $\M^{G}$ and $\C$ is given by explicitly showing the equivalence given the adjunction $\T \dashv \Sigma$.\footnote{In practice, given an adjunct pair it is often easier to show the equivalence directly rather than use the necessary and sufficient conditions.}

Throughout the paper, the ``\textbf{remark}'' paragraphs are intended as commentary and motivation for the material, and the technical  remarks  are not used in the subsequent development of the theory.  The extended remark at the end of \S \ref{SMCCMeas} provides the ``big picture'' how this research fits into the grand scheme of probability theory as (essentially) the theory of convex spaces.\footnote{The equivalence of the two categories, $\M^{\G}$ and $\Cvx$, implies every Giry algebra $(X, h)$ corresponds  to a convex space.  Hence the study of Giry algebras is the study of convex spaces.  The fact that $\Cvx$ is a complete, cocomplete, SMCC (a cosmos) provides the necessary framework for a vastly richer theory of probability than the current perspective.}

\section{Convex Spaces}

\subsection{Convex Space Structures}
 In defining convex spaces, we employ  the definitions given in B\"{o}rger and Kemp\cite{Borger}, so as to introduce the related categories of positively convex spaces, $\mathbf{PC}$, and superconvex spaces, $\mathbf{SC}$.\footnote{Our definitions are essentially verbatim from B\"{o}rger and Kemp, with the exception that we restrict the coefficients $\alpha_i$, as as used in the definition of an $\Omega$-algebra, to lie in the unit interval (rather than $\mathbb{R}$ or $\mathbb{C}$) and consequently do not require the modulus of these quantities, $|\alpha_i|$.  Although we note that the use of $\mathbb{C}$, or rather the (complex) unit disk, $\mathbb{D}$, yields an easy generalization of the theory that seems relevant to generalizing ``classical probability theory'' (which uses just the unit interval) to ``generalized (or quantum) probability theory''. 
  In this regard, the $2^{nd}$ footnote  suggest that at least some aspects of quantum probability theory may be be derived using Giry algebras (over $\mathbb{C}$ or $\mathbb{D})$, and (looking ahead) consequently probability measures  within that framework are weakly-averaging affine maps $\mathbb{D}^A \rightarrow \mathbb{D}$. (This course of thought requires that $\mathbb{D}$ be dense and codense in $\Cvx$, which we have not verified.)}  The alternative approach to defining convex spaces is that presented in Meng\cite{Meng}, which views $\Cvx$ as a single sorted theory.  For many purposes, such as showing completeness or that $\Cvx$ is a regular category, the definition as an algebraic theory is more useful.

 For $\Omega= \{  \mathbf{\alpha} \in \mbfI^{\mathbb{N}} \, | \, \sum_{i \in \mathbb{N}} \alpha_i \le 1\}$, 
 an $\Omega$-algebra is a set $A$ together with a map 
 \be \nonumber
 \begin{array}{ccc}
 \Omega &\rightarrow& \Set(A^{\mathbb{N}}, A) \\
 \mathbf{\alpha} & \mapsto&  A^{\mathbb{N}} \stackrel{\mathbf{\alpha_A}}{\longrightarrow} A
 \end{array}
 \ee
where $\mathbf{\alpha}_A$ is a set function.
A morphism of $\Omega$-algebras from $A$ to $B$ is a set map $A \stackrel{m}{\longrightarrow} B$ such that the diagram 
 \begin{equation}   \nonumber
 \begin{tikzpicture}[baseline=(current bounding box.center)]

          \node  (An) at  (0,0)  {$A^{\mathbb{N}}$};
          \node   (A)  at  (3,0)   {$A$};
         \node  (Bn)  at (0, -2)     {$B^{\mathbb{N}}$};
         \node   (B)   at   (3, -2)    {$B$};
         
         \draw[->,above] (An) to node {$\mathbf{\alpha}_A$} (A);
         \draw[->,left] (An) to node {$m^{\mathbb{N}}$} (Bn);
         \draw[->,right] (A) to node {$m$} (B);
         \draw[->,below] (Bn) to node {$\mathbb{\beta}_B$} (B);
                  
	 \end{tikzpicture}
 \end{equation}
commutes.  Obviously, the $\Omega$-algebras form a category, the composition of the morphisms being the set-theoretic composition.
 
 Let $\mathbf{a} \in A^{\mathbb{N}}$ denote a sequence in $A$.  Then for $\alpha \in \Omega$ and $A$ an $\Omega$-algebra with the mapping $\mathbf{\alpha}_A$ let $\mathbf{\alpha}_A(\mathbf{a}) = \sum_{i \in \mathbb{N}} \alpha_i a_i$ denote the value of the map $\alpha_A$ at $\mathbf{a}$.   
 An $\Omega$-algebra $A$ is called a positively convex space when $A \ne \emptyset$ and the following two axioms are satisfied.   
  \begin{enumerate}
 \item For $e_i^j \in \Omega$  the ``unit vector'',  
 \be \nonumber
 e_i^j = \left\{ \begin{array}{ll} 1&  \textrm{ iff }  i=j \\
 0 & \textrm{otherwise }
 \end{array} \right..
 \ee
 the following condition holds:
 \be \nonumber
 \displaystyle{ \sum_{i \in \mathbb{N}}} e_i^j a_i = a_j \quad \textrm{ for all }j \in \mathbb{N}, \, \textrm{ and all }\mathbf{a} \in A^{\mathbb{N}}.
 \ee
  
 \item  $\sum_{i \in \mathbb{N}} \alpha_i \bigg( \sum_{j \in \mathbb{N}} \beta_j^i a_j\bigg) = \sum_{j \in \mathbb{N}} (\sum_{i \in \mathbb{N}} \alpha_i \beta_j^i) a_j$, for all $\mathbf{\alpha} , \mathbf{\beta}^i \in \Omega$, and all $\mathbf{a} \in A^{\mathbb{N}}$.
 \end{enumerate}

The category of positively convex spaces, $\mathbf{PC}$, which is nonempty by hypothesis, has a nullary operator for each object $A$ in the category, called the zero (nullary) map,
\be \nonumber
1 \stackrel{0}{\longrightarrow} A
\ee
which is the image of the zero sequence $\mathbf{0} \in \Omega$, i.e., $A^{\mathbb{N}} \stackrel{\mathbf{0}}{\longrightarrow} A$  is a constant map determining  a unique element $0_A$.

A \emph{superconvex space} drops the hypothesis of being nonempty, but requires the limit of the countably infinite sum to have value $1$.  A \emph{convex space} is an $\Omega$-algebra restricted to sequences $\alpha \in \mbfI^{\mathbb{N}}$ such that only finitely many terms are nonzero, and the sum of those terms is one.  Hence we obtain the usual condition $\sum_{i=1}^n \alpha_i = 1$ associated with a ``convex sum'' $\sum_{i=1}^n \alpha_i a_i$.     
Note the category $\Cvx$ contains the object $\emptyset$  which is necessary for $\Cvx$ to be complete (and cocomplete). 
%(In $\Cvx$ the zero sequence $\mathbf{0}$ does not yield a set map $\Set(A^{\mathbb{N}}, A)$ since the terms do no sum to one.)  

 The unit interval $\mbfI$ can be viewed as an object in $\mathbf{PC}$, with the zero element being $0$.  Given any countable sequence $\{\alpha_i\}_{i=1}^{\infty}$ of terms in $\mbfI$ with $\lim_{n \rightarrow \infty} \{\sum_{i=1}^n \alpha_i \} \le 1$, we have upon taking any another countable sequence $\{p_i\}_{i=1}^{\infty}$ of element of $\mbfI$, the formal (infinite convex) sum
 \be \nonumber
 \sum_{i \in \mathbb{N}} \alpha_i p_i
 \ee 
where the elements $p_i$ are viewed as variables (or ``symbols'') of $\mbfI$ with the $\alpha_i$ being the  coefficients of those variables.  The elements $p_i$ need not be unique\footnote{These terms $p_i$ can be  interpretted  as values $p_i = P(m^{-1}(U_i))$ where $P$ is some probability measure on a convex space $A$, endowed with a $\sigma$-algebra, $\Sigma A$, and $\Sigma A \stackrel{m}{\longrightarrow} \mbfI$ is a measurable map. The $\{U_i\}_{i=1}^{\infty}$ partition the unit interval, with each $U_i$ itself a subinterval of $\mbfI$.}, and the limit of the infinite convex sum, upon evaluation (componentwise multiplication, $\alpha_i \cdot p_i$, and taking the limit of the sum), clearly gives a quantity in $\mbfI$.

The category of superconvex spaces, $\mathbf{SC}$, satisfies the following property, the proof of which can be found in \cite{Borger}.

\begin{thm} \label{superconvex} For $A$ a superconvex space and $a_0 \in A$, there exists a unique positively convex space  structure on $A$ with zero element $a_0$, such that the restriction of the operations to convex sums whose limit is one gives the original superconvex space.
\end{thm}

This result is useful for viewing a superconvex space as a positively convex space.  From the perspective of calculating limits of a sequence of real values in $\mbfI$, such as $\lim_{n \rightarrow \infty} \{\sum_{i=1}^n P(U_i)\}$, this category is what is used in practice.  Thus, while our main results focus on $\Cvx$, the connection between the various types of convex structures is relevant to computations and understanding, for example, how the representation of measurable functions via sequences of simple measurable functions fits into the big picture.

For subsequent reference, we note the following result.  Let \mbox{$\mathbb{R} = (-\infty, \infty)$} with the natural convex structure.  This convex structure extends to $\mathbb{R}_{\infty} = (-\infty,\infty]$, by defining for all $r \in \mathbb{R}$,  the convex sum 
\be \nonumber
\alpha \infty + (1-\alpha) r \stackrel{def}{=} \left\{ \begin{array}{ll} \infty & \textrm{for all }\alpha \in (0,1] \\ r & \textrm{for }\alpha=0 \end{array} \right.
\ee
 
\begin{thm}
The object  $\mathbb{R}_{\infty}$ is a coseparator in $\Cvx$.
\end{thm}
See \cite{Borger} for the proof.

\vspace{.2in}

\paragraph{\textit{Notation}}
  In the category $\Cvx$, the arrows  $A \stackrel{m}{\longrightarrow} B$  are referred to as affine maps, and preserve ``convex sums'',
\be \nonumber
m( (1-\alpha) a_1 + \alpha a_2 ) = (1-\alpha) m(a_1) + \alpha m(a_2) \quad \alpha \in \mbfI.
\ee
For brevity,  a convex sum is often denoted by 
 \be \nonumber
 a_1 +_{\alpha} a_2 \stackrel{def}{=} (1-\alpha) a_1 + \alpha a_2 \quad \alpha \in \mbfI.
 \ee

\vspace{.2in}
%Note that the proof shows that an uncountable union or intersection of measurable sets (with respect to the functor $\Sigma$), equations (\ref{PreDe}-\ref{DeMorgan}),  yield the two measurable sets which partition $A$.

\subsection{Boolean Subobjects of a Convex Space}   \label{BooleanSubobjects}

A subobject of a convex space $A$, say $A_0 \hookrightarrow A$, is called a \emph{Boolean} subobject when its set-theoretic complement $A_0^c$ is also a subobject of $A$. 
The functor $\mathbf{\Sigma}$ assigns to every convex space $A$ the measurable space $\mathbf{\Sigma}(A) = (A, \sa_A)$ \emph{generated} by the set of all Boolean  subobjects of $A$.   If $A_0$ is a Boolean subobject of $A$ then $A$ can be written as a sum (coproduct), $A = A_0 + A_0^c$.  

Using the convex space $\two=\{0,1\}$, which has the convex structure given  by 
\be \nonumber
(1-\alpha) \mathbf{0} + \alpha \one = \left\{ \begin{array}{ll} \mathbf{0} & \textrm{for all }\alpha \in [0,1) \\ \mathbf{1} & \textrm{otherwise} \end{array} \right..
\ee 
it is equivalent to say the $\sigma$-algebra of $A$ is generated by the set of all affine maps from $A$ into $\two$, $\Cvx(A, \two)$, since the preimages of such maps then yield the two complementary subobjects of $A$.\footnote{It is useful to this of the ``$\mathbf{0}$'' as ``$\infty$'', so that any nonzero quantity multiplied by $\infty$ gives back $\infty$, regardless of the contribution of the second term in the convex sum, $(1-\alpha)\mathbf{0} + \alpha \mathbf{1}$.}   In analogy with measurable spaces,  we will use the notation $\chi_{A_0}$ to denote elements of $\Cvx(A, \two)$.

The $\sigma$-algebra on a convex space $A$, generated by the set  $\Cvx(A, 2)$,  makes every affine map $A \stackrel{m}{\longrightarrow} B$  measurable since the composition of affine maps $A \stackrel{m}{\longrightarrow} B \rightarrow \two$ gives an element of the generating set for $\Sigma_A$.
It is elementary to verify that $\Sigma$ is functorial.

Viewing the unit interval as just a convex space,  all the subobjects of $\mbfI$ which are (sub)intervals of the form $(a,1]$ or $[a, 1]$ with $a>0$ have complements in $\mbfI$, and hence $\Sigma(\mbfI)$ yields the standard Borel $\sigma$-algebra on the unit interval.  Thus, applying the functor $\Sigma$ to the map $\T(\Sigma(\two)) \rightarrow \mbfI$ sending $\delta_0 +_{\alpha} \delta_1 \mapsto \alpha$, yields an isomorphism of measurable spaces.  Subsequently, rather than write $\Sigma(\mbfI)$, as well as  $\Sigma(\two)$, everywhere, we just write $\mbfI$ and $\two$, and let the context determine whether we are viewing these spaces as  convex spaces or measurable spaces.

Each element $a \in A$ determines a subobject of $A$ given by  
\be \nonumber
\langle \langle a \rangle \rangle = \{ b \in A \, | \, \exists c \in A, \, \exists \beta \in (0,1] s.t. \, \, a=c +_{\beta}b \}
\ee

\begin{lemma} \label{Lemma1} For each convex space $A$, and every $a \in A$ the map

\begin{equation}   \nonumber
 \begin{tikzpicture}[baseline=(current bounding box.center)]
 
  \node     (A) at  (0,0)  {$A$};
   \node   (2)  at   (6,0)   {$\two$};
  \node   (b)    at   (0,-1)   {$b$};
  \node   (c)    at   (5, -1)  {$\left\{ \begin{array}{ll} 1 & \textrm{iff }b \in \langle \langle a \rangle \rangle \\ 0 & \textrm{otherwise} \end{array} \right.$};
  \draw[->,above] (A) to node {$\chi_{\langle \langle a \rangle \rangle}$} (2);
  \draw[|->] (b) to node {$$} (c);

 \end{tikzpicture}
 \end{equation}
is affine, and hence $\langle \langle a \rangle \rangle$ and $\langle \langle a \rangle \rangle^c$ are both Boolean subobjects of $A$.
Moreover, we have the following properties:

\begin{enumerate}
\item If $\chi_{\langle \langle a \rangle \rangle}(b)=1$ then $\langle \langle b \rangle \rangle  \subset \langle \langle a\rangle \rangle$.
\item If $\chi_{\langle \langle a \rangle \rangle }(b) = \chi_{\langle \langle b \rangle \rangle}(a)=1$ then $\langle \langle a \rangle \rangle = \langle \langle b \rangle \rangle$.
\end{enumerate}
\end{lemma}
\noindent
The proof, which is straightforward, can be found in B\"{o}rger and Kemp\cite{Borger}.   Note that  $\langle \langle a \rangle \rangle$ and $\langle \langle a \rangle \rangle^c$ are therefore in the generating set for $\Sigma_A$ (and hence measurable under the functor $\Sigma$).

Given any  Boolean subobject $A_0 \hookrightarrow A$, if  $a \in A_0$ it follows that  $\gen{a} \subset A_0$  because if $b \in \langle \langle a \rangle \rangle$ then there exist a $d \in A$ and an $\alpha \in (0,1]$ such that
\be \nonumber
a = d +_{\alpha} b.
\ee
Since  $a \in A_0$, applying the affine map $\chi_{A_0}$ shows that
\be \nonumber
1 = \chi_{A_0}(a) = \chi_{A_0}(d) +_{\alpha} \chi_{A_0}(b).
\ee
and hence both $d$ and $b$ must also be in $A_0$, and hence  $\langle \langle a \rangle \rangle \subset A_0$.  The reverse equality follows from the fact $a \in \gen{a}$ for every $a \in A$.  Consequently we have, for every Boolean subobject $A_0$ 
\be \label{PreDe}
\bigcup_{a \in A_0} \langle \langle a \rangle \rangle = A_0.
\ee

Note that the two
constant maps,
\begin{equation}   \nonumber
 \begin{tikzpicture}[baseline=(current bounding box.center)]
 
  \node     (X) at  (0,0)  {$A$};
   \node   (TX)  at   (4,0)   {$\two$};
  \node    (x)   at   (0, -.8)   {$a$};
  \node    (dx) at   (4, -.8)  {$1$};
  
  \draw[->,above] (X) to node {$\chi_A$} (TX);
   \draw[|->] (x) to node {} (dx);

  \node     (Y) at  (6,0)  {$A$};
   \node   (TY)  at   (10,0)   {$\two$};
  \node    (y)   at   (6, -.8)   {$a$};
  \node    (dy) at   (10, -.8)  {$0$};
  
  \draw[->,above] (Y) to node {$\chi_{\emptyset}$} (TY);
   \draw[|->] (y) to node {} (dy);

 \end{tikzpicture}
 \end{equation}
are both in $\two^A$, and hence $A$ and $\emptyset$ are in the generating set for $\Sigma A$.  Moreover, 

\begin{lemma} \label{piSystem} The set of Boolean subobject of $A$  is closed under finite intersections. Hence the Boolean subobjects of any convex space form a $\pi$ system.
\end{lemma}
\begin{proof}
The set of Boolean subobject  is closed under finite intersection since if $A_0$ and $A_1$ are any two  Boolean subobjects of $A$, then $A_0 \cap A_1$ is a subobject of $A$ with Boolean complement $(A_0\cap A_1^c) \cup (A_0^c \cap A_1) \cup (A_0^c \cap A_1^c)$.  Thus it is a $\pi$-system. 

\end{proof}

\vspace{.1in}

\subsection{The density of the unit interval}  \label{density}
The full subcategory of $\Cvx$ consisting of the single object, the unit interval $\mbfI=[0,1]$, with all its endomorphisms, is  a left adequate (dense) subcategory of $\Cvx$.  This property follows from a theorem due to Isbell\cite{Isbell}   showing that if a category $\mathcal{C}$ is a full subcategory of algebras with operations at most $n$-ary, then the free algebra on $n$ generators, with all its endomorphisms  is a dense subcategory.\footnote{Isbell used the terminology of \emph{left-adequate} rather than dense. We subsequently also use this terminology in our subsequent discussion.}  %\footnote{Recall, a subcategory $\mathcal{R}$ is an adequate (dense) subcategory of $\mathcal{C}$ provided the ``truncated Yoneda map'', $\mathcal{C} \longrightarrow \Set^{\mathcal{R}^{op}}$ is still a full embedding.  In a tribute to John Isbell, Lawvere\cite{LawIsbell} provides a nice description of this concept.}   
Meng\cite{Meng} treats the theory of convex spaces as the full subcategory of the category of $\K-$modules, with the same objects  but only the idempotent operations.  An $n$-ary operation $A^n \stackrel{\phi}{\longrightarrow} A$ is idempotent provided that $\phi \circ \Delta_n = id_A$, where $\Delta_n$ is the diagonal mapping.  Using this the theory of convex spaces can be defined inductively  using only the  operations of maximum arity $2$ to define the theory.\footnote{This is, of course, similiar to the situation in ordinary addition where knowing the single operation ``+''  suffices to define the meaning of $(a+b)+c = a+(b+c)$.} Then, since the free object on  $2$ elements in $\Cvx$ is the unit interval $\mbfI$, it follows that the subcategory $\I$, consisting of the single object $\mbfI$ along with all the affine endomorphisms $\Cvx(\mbfI, \mbfI)$, is adequate in $\Cvx$.

 The slice category $\I/A$ has as objects affine maps $\mbfI  \stackrel{n}{\longrightarrow} A$, and as arrows, affine maps $\mbfI \stackrel{\gamma}{\longrightarrow} \mbfI$ such that the diagram 
%\begin{figure}[H]
\begin{equation}   \nonumber
 \begin{tikzpicture}[baseline=(current bounding box.center)]

          \node  (I1) at  (0,0)  {$\mbfI$};       \node   (I2)  at  (3,0)   {$\mbfI$};
         \node  (A)  at (1.5,-1.5)    {$A$};
         
         \draw[->,above] (I1) to node {$\gamma$} (I2);
         \draw[->,left] (I1) to node {$n$} (A);
         \draw[->,right] (I2) to node {$m$} (A);
                  
	 \end{tikzpicture}
 \end{equation}
% \caption{An $\I/A$ morphism, $n \stackrel{\gamma}{\longrightarrow} m$.}
% \end{figure}
 \noindent
commutes.  Left adequacy (density) of $\I$ means  every  convex space $A$ is a canonical limit
\be \nonumber
A \cong colim \bigg( \I/A \stackrel{\pi}{\longrightarrow} \I \hookrightarrow \Cvx \bigg),
\ee
where $\pi$ is the projection map (functor).  Left adequacy of $\I$, using the definition given by Isbell (extended to the cosmos\footnote{A SMCC which is complete and cocomplete.} $\Cvx$ rather than $\Set$), is that the restricted Yoneda embedding 
\begin{equation}   \nonumber
 \begin{tikzpicture}[baseline=(current bounding box.center)]

          \node  (C) at  (0,0)  {$\Cvx$};       \node   (CI)  at  (5,0)   {$\Cvx^{\I^{op}}$};
         \node  (A)  at (0,-.8)    {$A$};      \node  (hA) at   (5, -.8)  {$\Cvx(\cdot, A)$}; 
         \node   (m) at   (0, -1.8)  {$A \stackrel{m}{\longrightarrow} B$};  \node   (hm) at  (5, -1.8)  {$\Cvx(\cdot, A) \stackrel{\Cvx(\cdot,m)}{\longrightarrow} \Cvx(\cdot,B)$};
         
         \draw[->,above] (C) to node {$Y|$} (CI);
         \draw[|->] (A) to node {} (hA);
         \draw[|->,right] (m) to node {} (hm);
                  
	 \end{tikzpicture}
 \end{equation}
is a full and faithful functor.   The equivalence between these two definitions follows from the fact that every object $\mathcal{F} \in_{ob} \Cvx^{\I^{op}}$ is a colimit of representables of the functor category $\Cvx^{\I^{op}}$.

The maps of the slice category $\I/A$,  $\psi \in \Cvx(\mbfI,\mbfI)$, can be characterized via a ``scaling'' parameter $s$, and a  ``translation'' parameter $t$,   which defines  the transformation given, for all $\alpha \in \mbfI$, by
\be \nonumber
\langle s,t \rangle (\alpha) = s \alpha + t  \quad \quad \forall s \in [-1,1],  \, \, \forall  t \in [0,1],  \, \, \textrm{satisfying }0 \le s + t \le 1,
\ee
and such a transformation factors as a composite, $\langle s, t \rangle = \langle 1, t \rangle \circ \langle s, 0 \rangle$.  
%We refer to the affine maps in $\Aff$ of the form $\langle 0, t \rangle$ as \emph{translation} maps, and elements of the form $\langle s, 0 \rangle$ as \emph{scaling} maps.  (We use the word maps, transformations, and elements interchangably when referring to elements of $\Aff$.)

An object of $\I/A$ is called a  \textbf{path map} in $A$ and denoted
\begin{equation}   \nonumber
 \begin{tikzpicture}[baseline=(current bounding box.center)]
  \node  (I) at  (0,0)   {$\mbfI$};
  \node   (I2) at  (4, 0)  {$A$};
  \node   (x)  at  (0,-.8)  {$\alpha$};
  \node   (fx) at   (4,-.8)  {$a_1 +_{\alpha}  a_2$};
  
  \draw[->,above] (I) to node {$\lbrack a_1, a_2 \rbrack$} (I2);
  \draw[|->] (x) to node {} (fx);

 \end{tikzpicture}
 \end{equation}
sending the two generators of $\mbfI$, denoted by $\underline{0}$ and $\underline{1}$,  to \emph{any}  two elements $a_1, a_2 \in A$,
\be \nonumber
\begin{array}{ccc}
\underline{0} & \mapsto & a_1 \\
\underline{1} & \mapsto & a_2
\end{array},
\ee
and the convex structure of $A$ itself then ``determines all the intermediate points'' on that path.   Thus, every affine map $\mbfI \rightarrow A$ is a path map in $A$, and the slice category $\I/A$ gives the basic figure types, which are the path maps, and all the incidence relationships between these path maps which completely characterize the convex space. 

\vspace{.2in}
The set of all path maps into the convex space $\two$
 are depicted in the diagram
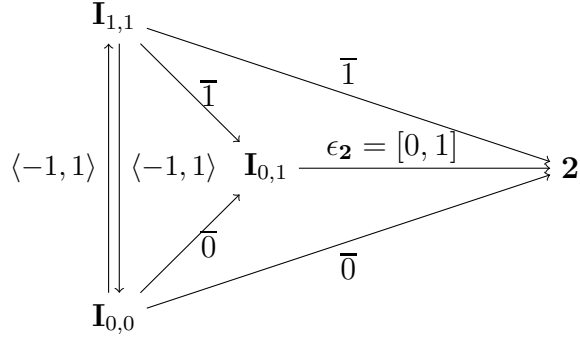
\begin{equation}   \nonumber
 \begin{tikzpicture}[baseline=(current bounding box.center)]
  \node  (I) at  (1,0)   {$\mbfI_{0,1}$};
  \node   (I0) at  (-1, -2)  {$\mbfI_{0,0}$};
  \node  (I1)  at   (-1, 2)  {$\mbfI_{1,1}$};
%  \node   (1)  at   (3,-2)   {$\one$};
%  \node   (12) at  (3,2)    {$\one$};

  \node   (2) at  (5,0) {$\two$};
  \draw[->,below] (I0) to node {$\overline{0}$} (2);
   \draw[->,above] (I) to node [xshift=-12pt,yshift=-2pt] {$\epsilon_{\two}=\lbrack 0,1\rbrack$}(2);
  \draw[->,above] (I1) to node {$\overline{1}$} (2);
  \draw[->,left] ([xshift=-2pt] I0.north) to node {\small{$\langle -1, 1\rangle$}} ([xshift=-2pt] I1.south);
 \draw[->,right] ([xshift=2pt] I1.south) to node {\small{$\langle -1, 1 \rangle$}} ([xshift=2pt] I0.north);
 \draw[->,right] (I1) to node {$\overline{1}$} (I);
 \draw[->,right] (I0) to node {$\overline{0}$} (I);
% \node   (c)  at   (7,1)    {Note $[1,0]$ is the complementation map$};    

 \end{tikzpicture}
 \end{equation}
 \noindent
Both of the constant maps, $\mbfI_{0,0} \stackrel{\overline{0}}{\longrightarrow} \mbfI$ and $\mbfI_{1,1} \stackrel{\overline{1}}{\longrightarrow} \mbfI$, have the identity map  (not shown) as a section.   
The map $\epsilon_{\two}$ is given explicitly by
\begin{equation}   \label{def2}
 \begin{tikzpicture}[baseline=(current bounding box.center)]

         \node  (I2)  at   (-1,-4)     {$\mbfI$};
         \node   (22) at   (3.5,-4)    {$\two$};
         \draw[->, above] (I2) to node {$\epsilon_{\two}$} (22);
         
%         \node   (c)  at   (1, -1)   {in $\M$};
%         \node    (d)  at  (6.5,-1)   {in $\Cvx$};
        \node  (a)  at   (-1, -4.8)   {$\alpha$};
         \node    (e)  at    (3.5,-4.8)    {$[0,1](\alpha) = \left\{ \begin{array}{ll} 0 & \textrm{for all }\alpha \in [0,1) \\ 1 & \textrm{for }\alpha=1 \end{array} \right..$};
\draw[|->] (a) to node {} (e);
 \end{tikzpicture}
 \end{equation}
 \noindent
We next use this map $\epsilon_{\two}$ to show that $\Sigma A$ is separated.

\section{The symmetric monoidal closed structure of $\M$.}  \label{SMCCMeas}

Throughout this section, $X$ and $Y$ denote measurable spaces.
The category $\M$  is a SMCC with the tensor product $X \otimes Y$  defined by the coinduced (final) $\sigma$-algebra such that all the  graph functions
\be  \nonumber
\begin{array}{ccccc}
\Gamma_f &:& X & \longrightarrow & X \times Y \\
&:& x & \mapsto & (x,f(x))
\end{array}
\ee 
for $X \stackrel{f}{ \longrightarrow} Y$ a measurable function, as well as the graph functions
\be  \nonumber
\begin{array}{ccccc}
\Gamma_g &:& Y & \longrightarrow & X \times Y \\
&:& y & \mapsto & (g(y),y)
\end{array}
\ee
for $Y \stackrel{g}{\longrightarrow} X$ a measurable function, are measurable. 
 
Let $Y^X$ denote the set of all measurable functions from $X$ to $Y$ endowed with the $\sigma$-algebra induced by the set of all point evaluation maps\footnote{This is equivalent to saying $Y^X$ has the product $\sigma$-algebra induced by the coordinate projection maps onto $Y$.} 
\be \nonumber
\begin{array}{ccc}
Y^X &\stackrel{ev_x}{\longrightarrow}& Y \\
\ulcorner f \urcorner & \mapsto & f(x) 
\end{array}
\ee 

   Because the $\sigma$-algebra structure on tensor product spaces is defined  such that the  graph functions are all measurable, it follows in particular the constant graph functions \mbox{$X \stackrel{\Gamma_{\ulcorner f \urcorner}}{\longrightarrow} X \otimes Y^X$} sending $x \mapsto (x,\ulcorner f \urcorner)$ are measurable.
   % (The graph function symbol $\Gamma_{\cdot}$ is overloaded and will need to be specified directly, both the domain and codomain, when the context is not clear.)

Define the evaluation function
 \be \nonumber
 \begin{array}{ccc}
 X \otimes Y^X & \stackrel{ev_{X,Y}}{\longrightarrow}& Y \\
 (x,\ulcorner f \urcorner) & \mapsto & f(x)
 \end{array}
 \ee
and observe that for every $\ulcorner f \urcorner  \in Y^X$ the right hand diagram in the $\M$ diagrams

\be  \nonumber
 \begin{tikzpicture}[baseline=(current bounding box.center)]
 	\node	(X)	at	(0,-1.8)              {$X \cong X \otimes 1$};
	\node	(XY)	at	(0,0)	               {$ X \otimes Y^X$};
	\node	(Y)	at	(3,0)               {$Y$};
         \node         (1)    at      (-4,-1.8)             {$1$};
         \node         (YX) at      (-4,0)              {$Y^X$};

	\draw[->, left] (X) to node  {$\Gamma_{\ulcorner f \urcorner}\cong  Id_X \otimes \ulcorner f \urcorner$} (XY);
	\draw[->,below, right] (X) to node [xshift=3pt,yshift=-1pt] {$f$} (Y);
	\draw[->,above] (XY) to node {$ev_{X,Y}$} (Y);
	\draw[->,left] (1) to node {$\ulcorner f \urcorner$} (YX);

 \end{tikzpicture}
 \ee
\noindent 
is commutative as a set mapping, $f = ev_{X,Y} \circ \Gamma_{\ulcorner f \urcorner}$.   
%The subscript notation associated with these evaluation maps is usually dropped but it is important to realize that every different $(X,Y)$ pair yields a different arrow in $\M$.
By rotating the above diagram and also considering the constant graph functions $\Gamma_{x}$  the right hand side of the  diagram 
\be \nonumber 
 \begin{tikzpicture}[baseline=(current bounding box.center)]
        \node          (X)    at      (-3,0)             {$X$};
 	\node	(YX)	at	(3,0)              {$Y^X$};
	\node	(XY)	at	(0,0)	               {$ X \otimes Y^X$};
	\node	(Y)	at	(0,-1.8)               {$Y$};

         \draw[->,above] (X) to node {$\Gamma_{\ulcorner f \urcorner}$} (XY);
         \draw[->,left] (X) to node [xshift=-7pt] {$f$} (Y);
	\draw[->, above] (YX) to node  {$\Gamma_{x}$} (XY);
	\draw[->,right] (YX) to node [xshift=5pt,yshift=0pt] {$ev_x$} (Y);
	\draw[->,right] (XY) to node {$ev_{X,Y}$} (Y);

 \end{tikzpicture}
 \ee
\noindent 
also commutes for every $x \in X$.
Since $f$ and $\Gamma_{\ulcorner f \urcorner}$ are measurable, as are $ev_{x}$ and $\Gamma_{x}$,  it follows by
the elementary result on coinduced $\sigma$-algebras 
   
\begin{lemma} \label{coinduced} Let the $\sigma$-algebra of  $\, Y$ be coinduced by a collection of maps \mbox{$\{f_j : X_j \rightarrow Y \}_{j \in J}$}.   Then a function $g:Y \rightarrow Z$ is measurable if and only if the composition $g \circ f_j$ is measurable for each $j \in J$.
\end{lemma}
\noindent
that $ev_{X,Y}$ is measurable because the graph functions generate the $\sigma$-algebra of $X \otimes Y^X$.  
 
More generally, given any measurable function $f:X \otimes Z \rightarrow Y$ there exists a unique measurable map $\tilde{f} : Z  \rightarrow  Y^X$  defined by $\tilde{f}(z) = \ulcorner f(\cdot,z) \urcorner: 1 \rightarrow Y^X$ where $f(\cdot,z): X \rightarrow Y$ sends $x \mapsto f(x,z)$.
This map $\tilde{f}$ is measurable because the $\sigma$-algebra is generated by the \emph{point evaluation} maps $ev_x$ and the diagram

\be  \nonumber
 \begin{tikzpicture}[baseline=(current bounding box.center)]
        \node          (XZ)    at      (3,-1.8)             {$X \otimes Z$};
 	\node	(YX)	at	(0,0)              {$Y^X$};
	\node	(Y)	at	(3,0)               {$Y$};
	\node	(Z)	at	(0,-1.8)	               {$Z$};

         \draw[->,above] (YX) to node {$ev_x$} (Y);
         \draw[->,left] (Z) to node [xshift=-3pt] {$\tilde{f}$} (YX);
	\draw[->, above] (Z) to node [xshift=5pt] {$\Gamma_{x}$} (XZ);
	\draw[->,right] (XZ) to node [xshift=5pt,yshift=0pt] {$f$} (Y);
	\draw[->,right,dashed] (Z) to node {$$} (Y);

 \end{tikzpicture}
\ee
\noindent 
commutes so that $\tilde{f}^{-1}( ev_x^{-1}(B)) = (f \circ \Gamma_{x})^{-1}(B) \in \sa_Z$.

Conversely given any measurable map $g  : Z \rightarrow Y^X$ it follows the composite 
\mbox{$ev_{X,Y} \circ (Id_X \otimes g)$}
is a measurable map. This determines a bijective correspondence 
\be \nonumber
\M(X \otimes Z,Y) \cong \M(X,Y^Z).
\ee

\paragraph{Double dualization into the unit interval $\mbfI$}    As the function space $\mbfI^X$ has the product \mbox{$\sigma$-algebra} it follows that each of the 
 point evaluation maps $\mbfI^X \stackrel{ev_x}{\longrightarrow} \mbfI$, for every $x \in X$,  is measurable.

\begin{lemma}  \label{measI}
Given any measurable space $X$ the double dual mapping\footnote{In this diagram and those to follow we abuse notation following the doctrine of  expressing the mapping into a function space not as the name of an element, like $\ulcorner ev_x \urcorner \in \mbfI^{\mbfI^X}$ for the given map $\eta_X(x)$,  but rather as the morphism corresponding to the named element.  The dashed arrow notation is employed to make it easier to read given the multiple arrows involved.}
\begin{equation}   \nonumber
 \begin{tikzpicture}[baseline=(current bounding box.center)]
         \node  (PA)  at (0,0)    {$X$};
         \node  (I)     at    (4,0)     {$\mbfI^{\mbfI^X}$};

         \node  (PA2)  at (0,-.7)    {$x$};
         \node  (I2)     at    (4,-.7)     {$\mbfI^X \stackrel{ev_x}{\longrightarrow} \mbfI$};
         
	\draw[->,above] (PA) to node {$\eta_{X}$} (I);
	\draw[|->,densely dashed] (PA2) to node {} (I2);
	
\end{tikzpicture}
 \end{equation}
 \noindent
is a measurable function.
\end{lemma}
\proof
Since the functions $\{ev_f\}_{f \in \M(X,\mbfI)}$ generate $\sa_{\mbfI^{\mbfI^X}}$ it suffices to show that \\
\mbox{$\eta^{-1}_{X}(ev_f^{-1}(U)) \in \sa_X$ for $U \in \sa_{\mbfI}$}.  But this set is just $f^{-1}(U)$ which is measurable since $f$ is measurable.
\endproof

\begin{remark}
We will subsequently show that when $X=\Sigma A$, for some convex space $A$, then the convex space $\mbfI^{\mbfI^A}$, restricted to the weakly averaging affine maps, $\mbfI^{\mbfI^A}|_{wa}$, is an  isomorphism, $A \cong \mbfI^{\mbfI^A}|_{wa}$.  Consequently every such weakly averaging affine map $\mbfI^A \stackrel{P}{\longrightarrow} \mbfI$ which (algebraically) we view as probability measures on $A$, must correspond to a point evaluation, i.e., $P = ev_a$ for some $a \in A$.  Of course, this is not true when $X \ne \Sigma A$ for some convex space $A$.  

In the case $\mbfI^{X} \stackrel{P}{\longrightarrow} \mbfI$ where $P$ is a weakly averaging affine (measurable) function, we still can view $P$ algebraically as a map between convex spaces, since the function space $\mbfI^X=\M(X, \mbfI)$ has a convex structure defined pointwise, $(f+_{\alpha} g)(x) = f(x) +_{\alpha} g(x)$.  It is this perspective that leads us to the result that $\G(X) \cong \mbfI^{\mbfI^X}|_{wa}$.  This equivalence corresponds to a natural isomorphism of monads, $\G \cong \mbfI^{\mbfI^{\bullet}}|_{wa}$, where the latter monad is the (codomain restricted) double dualization monad into $\mbfI$.   This in turn leads to the perspective that integration is just the evaluation function $\mbfI^{\mbfI^X}|_{wa} \otimes \mbfI^X \stackrel{ev}{\longrightarrow} \mbfI$, which using the SMCC structure of $\M$, is therefore a measurable function, as are the weakly averaging affine maps (which correspond to probability measures).  These ideas are discussed in \cite{Sturtz}.  This paper extends that perspective leading to the idea that \emph{probability theory} should be viewed within the framework of  the category $\Cvx$, where the basic ideas of Bayesian probability, (1)  Bayesian models and (2) calculation of inference maps,  can be extended to yield a richer modeling framework as well as the development of  computationally more efficient inference algorithms.  The latter idea coming from the idea that since $\Cvx$ has a much richer (nicer) categorical structure than $\M_{\G}$, computations should be easier than in $\M_{\G}$ where we traditionally do Bayesian calculations.\cite{Culbertson}  (The need for ``quotient spaces'' in the computation of inference maps is the motivation for our development in this paper. $\Cvx$ has quotients, $\M_{\G}$ lacks quotients.) We hope to give a formal  justification for the remark \emph{computationally more efficient} in a future paper.
\end{remark}

 \section{The measurable space $\Sigma A$ is separated}  \label{separated}
 Given any measurable space $X$, we say it is separated if and only if for any two distinct points $x_1, x_2$ in the space there is a measurable subset $U \in \Sigma_X$ such that $x_1 \in U$ while $x_2 \not \in U$.
 
  To prove the space $\Sigma A$ is separated  we start with the fact that the object $\mathbb{R}_{\infty}$ is a coseparator in $\Cvx$.  Hence  for any two distinct points $a_1, a_2 \in A$, there exist an affine  map $A \stackrel{m}{\longrightarrow} \mathbb{R}_{\infty}$ which separates the pair, $m(a_1) \ne m(a_2)$.  The composite map
\begin{equation}   \label{def2}
 \begin{tikzpicture}[baseline=(current bounding box.center)]

         \node  (I)  at   (0,0)     {$\mbfI$};
         \node   (A) at   (3, 0)    {$A$};
         \node   (R) at  (6,0)   {$\mathbb{R}_{\infty}$};
         
         \draw[->, above] (I) to node {$[a_1, a_2]$} (A);
         \draw[->,above] (A) to node {$m$} (R);

 \end{tikzpicture}
 \end{equation}
\noindent
is a path in $\mathbb{R}_{\infty}$ with $m(a_1) \ne m(a_2)$.  Without loss of generality, we can assume that $m(a_1)< m(a_2)$, and hence $m(a_1) < \infty$.  Now observe, that for any $\omega \in \mathbb{R}_{\infty}$, the two subsets $(-\infty, \omega)$ and $[\omega, \infty]$ are complementary subobjects of $\mathbb{R}_{\infty}$.  In otherwords, $\mathbb{R}_{\infty}$ can be viewed as the sum (coproduct) of the two subobjects, 
\be  \nonumber
\begin{tikzpicture}[baseline=(current bounding box.center)]
\node  (R)  at   (0,0)  {$\mathbb{R}_{\infty}$};
\node   (l)  at    (-3,0)  {$(-\infty, \omega)$};
\node    (r)  at   (3,0)    {$ [\omega, \infty]$};
\draw[>->,above] (l) to node {$$} (R);
\draw[>->,above] (r) to node {$$} (R);
\end{tikzpicture}
\ee
with the inclusion maps.  Since this space is a coproduct, any map to $\two$ defined on the components yields a unique map from $\mathbb{R}_{\infty}$ to $\two$,
\be  \nonumber
\begin{tikzpicture}[baseline=(current bounding box.center)]
\node  (R)  at   (0,0)  {$\mathbb{R}_{\infty}$};
\node   (l)  at    (-3,0)  {$(-\infty, \omega)$};
\node    (r)  at   (3,0)    {$ [\omega, \infty]$};
\node    (2)  at   (0, -2)  {$\two$};
\draw[->, left] (l) to node [yshift=-3pt] {$\overline{0}$} (2);
\draw[->, right] (r) to node [yshift=-3pt] {$\overline{1}$} (2);
\draw[->, right] (R) to node [yshift=6pt] {$\chi_{[\omega,\infty]}$} (2);
\draw[>->,above] (l) to node {$$} (R);
\draw[>->,above] (r) to node {$$} (R);
\end{tikzpicture}
\ee
\noindent
That unique map is the characteristic function $\chi_{(-\infty, \omega)}$, which is in fact an affine map because for $x \in (-\infty, \omega)$ and $y \in [\omega, \infty]$ we have, 
\be \nonumber
x+_{\alpha} y \in (-\infty, \omega) \quad \quad \textrm{for all }\alpha \in [0,1)
\ee
and hence 
\be \nonumber
\chi_{[\omega,\infty]}(x +_{\alpha} y) = 0
\ee
whereas
\be \nonumber
\chi_{[\omega,\infty]}(y) +_{\alpha} \chi_{[\omega,\infty]}(x) = 0 +_{\alpha} 1 = 0.
\ee
Clearly, if both $x$ and $y$ are both in either component of $\mathbb{R}_{\infty}$ then $\chi_{[\omega, \infty]}$ is the constant function on that component, and hence affine.

Choose $\omega= m(a_2)$.  The map $A \stackrel{m}{\longrightarrow} \mathbb{R}_{\infty}$ which separates the two points, $a_1$ and $a_2$,  in conjunction with the map $\chi_{[m(a_2),\infty]}$, in turn yields a Boolean subobject pair of objects of $A$, with corresponding insertion maps, 
 \be  \nonumber
\begin{tikzpicture}[baseline=(current bounding box.center)]
\node  (A)  at    (0,2)  {$A$};
\node  (A1) at   (-3.5, 2)  {$m^{-1}((-\infty, m(a_2))$};
\node   (A2) at   (3.5, 2)   {$m^{-1}([m(a_2), \infty])$};
\draw[>->,above] (A1) to node {$\iota_1$} (A);
\draw[>->,above] (A2) to node {$\iota_2$} (A);

\node  (R)  at   (0,0)  {$\mathbb{R}_{\infty}$};
\node   (l)  at    (-3.5,0)  {$(-\infty, m(a_2))$};
\node    (r)  at   (3.5,0)    {$ [m(a_2), \infty]$};
\node    (2)  at   (0, -2.5)  {$\two$};
\draw[->, left] (l) to node [yshift=-3pt] {$\overline{0}$} (2);
\draw[->, right] (r) to node [yshift=-3pt] {$\overline{1}$} (2);
\draw[->, right] (R) to node [yshift=6pt] {\small{$\chi_{[m(a_2),\infty]}$}} (2);
\draw[>->,above] (l) to node {$$} (R);
\draw[>->,above] (r) to node {$$} (R);
\draw[->,right] (A) to node {$m$} (R);
\draw[->,left] (A1) to node {$m \circ \iota_1$} (l);
\draw[->,right] (A2) to node {$m \circ \iota_2$} (r);
\end{tikzpicture}
\ee
\noindent
where $m_1 = m \circ \iota_1$ and $m_2=m \circ \iota_2$ are the two restriction maps associated with $m$.
The composite map 
\be \nonumber
\chi_{m^{-1}([m(a_2), \infty])} = \chi_{[m(a_2),\infty]} \circ m
\ee
 yields the commutative diagram
 \begin{equation}   \label{def2}
 \begin{tikzpicture}[baseline=(current bounding box.center)]

         \node  (I)  at   (3,2)     {$\mbfI$};
         \node   (A) at   (3, 0)    {$A$};
         \node   (R) at  (3,-2)   {$\mathbb{R}_{\infty}$};
         \node    (2)  at  (7,0)    {$\two$};
         
         \draw[->, left] (I) to node {$[a_1, a_2]$} (A);
         \draw[->,left] (A) to node {$m$} (R);
         \draw[->,above] (A) to node [xshift=-3pt] {$\chi_{m^{-1}([m(a_2), \infty])}$} (2);
         \draw[->,right] (R) to node [yshift=-3pt]{$\chi_{[m(a_2), \infty]}$} (2);
         \draw[->,above] (I) to node {$\epsilon_{\two}$} (2);
 \end{tikzpicture}
 \end{equation}
\noindent
Since the Boolean subobjects of $A$ generate the $\sigma$-algebra on $A$, it follows that  the Boolean subobject pair of $A$, $m^{-1}((-\infty, m(a_2))$ and $m^{-1}([m(a_2),\infty])$, are both measurable in $\Sigma A$, and that they  separate the pair of distinct points, $\{a_1, a_2\}$, 
\be \nonumber  
a_1 \not \in m^{-1}([m(a_2),\infty]) \quad \quad \textrm{while } \quad \quad a_2 \in m^{-1}([m(a_2), \infty]).
\ee
All told, $\Sigma A$ is a separated measurable space.

\begin{cor} \label{ddMap}
 In $\Cvx$, the double dualization map
\begin{equation}   \nonumber
 \begin{tikzpicture}[baseline=(current bounding box.center)]
 
  \node     (A) at  (0,0)  {$A$};
   \node   (IIA)  at   (5,0)   {$\mbfI^{\mbfI^A}$};
  \node    (a)   at   (0, -.8)   {$a$};
  \node    (IA) at   (4, -.8)  {$\mbfI^A$};
  \node    (I)    at    (6, -.8)  {$\mbfI$};
  
  \draw[->,above] (A) to node {$\widetilde{\eta}_A$} (IIA);
   \draw[|->,dashed] (a) to node {} (IA);
   \draw[->,above] (IA) to node {$ev_a$} (I);
 \end{tikzpicture}
 \end{equation}
\noindent
is monic (=injective).
\end{cor}
We use the notation $\widetilde{\eta}_A$ to distinguish this map from the unit of the Giry monad, $\eta$.  The maps $\widetilde{\eta}_A$ form the components of a natural transformation $\widetilde{\eta}$, which is the unit of the (well known) double dualization monad into $\mbfI$.
\begin{proof}
Let $a_1, a_2$ be two distinct points of $A$.  We must show there exist a map $A \stackrel{\hat{m}}{\longrightarrow} \mbfI$ which separates the pair $\{ev_{a_1}, ev_{a_2}\}$,
\be \nonumber
ev_{a_1}(\ulcorner \hat{m} \urcorner) = \hat{m}(a_1) \ne \hat{m}(a_2) = ev_{a_2}(\ulcorner \hat{m} \urcorner).
\ee
Since $\mathbb{R}_{\infty}$ is a coseparator, let $m$ be the affine map $A \stackrel{m}{\longrightarrow} \mathbb{R}_{\infty}$ which coseparates $a_1$ and $a_2$.  Thus by the preceding constructions, the affine map
$\chi_{m^{-1}([m(a_2), \infty])}$ satisfies the property that
\be \nonumber
\chi_{m^{-1}([m(a_2), \infty])}(a_1)=0 \quad \quad \textrm{ while }  \quad \quad \chi_{m^{-1}([m(a_2), \infty])}(a_2)=1.
\ee
and we have the commutative diagram
%Applying the functor $\Sigma$ to this affine map makes it measurable, and applying the functor $\T$ to $\Sigma \chi_{m^{-1}([m(a_2), \infty])}$ yields the affine map
\begin{equation}   \nonumber
 \begin{tikzpicture}[baseline=(current bounding box.center)]
 
  \node     (A) at  (0,0)  {$A$};
  \node    (IIA)  at  (0, 2)  {$\mbfI^{\mbfI^A}$};
  \node    (III)  at   (6, 2)  {$\mbfI$}; 
   \node   (I)  at   (6,0)   {$\two$};
   
  \draw[->,below] (A) to node {$\chi_{m^{-1}([m(a_2),\infty])}$} (I);
  \draw[>->,left] (A) to node {$\widetilde{\eta}_A$} (IIA);
  \draw[->,above] (IIA) to node {$ev_{m}$} (III);
  \draw[->>,right] (III) to node {$\epsilon_{\two}$} (I);
  \draw[->,above,dashed] (A) to node [xshift=-10pt]{$ev_m \circ \eta_A$} (III);
  
   \end{tikzpicture}
 \end{equation}
\noindent
The composite map $ev_m \circ \widetilde{\eta}_A$ is a ``lifting'' of the characteristic map $\chi_{m^{-1}([m(a_2), \infty])}$.  Taking $\hat{m} = ev_m \circ \widetilde{\eta}_A$ we obtain
\be \nonumber
\begin{array}{lcl}
ev_{a_1}(\ulcorner ev_m \circ \widetilde{\eta}_A \urcorner) &=& ev_m(\ulcorner \widetilde{\eta}_A(a_1) \urcorner) \\
&=& ev_m( \ulcorner ev_{a_1} \urcorner) \\
&=& m(a_1) 
\end{array}
\ee
Similarly, $ev_{a_2}(\ulcorner ev_m \circ \widetilde{\eta}_A \urcorner) = m(a_2)$.  Since $m(a_1) \ne m(a_2)$ the map $\hat{m} = ev_m \circ \widetilde{\eta}_A$ separates the pair $\{ev_{a_1}, ev_{a_2}\}$, i.e., proves they are distinct maps.

\end{proof}

\section{The codensity of $\I$ in $\Cvx$} \label{codensity}

The convex space  $\I$ is codense (right-adequate) in $\Cvx$ when  the restricted dual Yoneda embedding \mbox{$\Cvx^{op} \hookrightarrow \Cvx^{\I}$} is full and faithful.  The property of being  faithful implies that the double dualization map into $\mbfI$ is monic, which is the preceding corollary.  Hence it only remains to show that the restricted dual Yoneda embedding is full.

The property of being full requires that any affine map
 $\mbfI^A \stackrel{P}{\longrightarrow} \mbfI$ which makes the diagram 
 
\begin{equation}   \nonumber
 \begin{tikzpicture}[baseline=(current bounding box.center)]
 
  \node     (IA) at  (0,0)  {$\mbfI^A$};
  \node      (I1) at  (3,0)  {$\mbfI$};
  \node     (IA2)  at  (0, -2) {$\mbfI^A$};
  \node     (I2)  at   (3, -2)   {$\mbfI$};
  \draw[->,above] (IA) to node {$P$} (I1);
  \draw[->,left] (IA) to node {$\langle s,t \rangle^A$} (IA2);
  \draw[->,below] (IA2) to node {$P$} (I2);
  \draw[->,right] (I1) to node {$\langle s, t \rangle$} (I2);
  
   \node   (m)  at   (5,0)   {$m$};
  \node    (stm)   at   (5, -2)   {$\langle s,t \rangle \circ m$};
  \node    (Pstm) at   (7.6, -2)  {$P(\langle s,t \rangle \circ m)$};
  \node    (Pm)    at    (10, -0)  {$P(m)$};
  \node     (p)      at    (10, -2)   {$=s P(m) +t$};
    
   \draw[|->] (m) to node {} (Pm);
   \draw[|->] (Pm) to node {} (p);
   \draw[|->] (m) to node {} (stm);
   \draw[|->] (stm) to node {} (Pstm);
 \end{tikzpicture}
 \end{equation}
\noindent
commute, implies that 
their exist some $a \in A$ such that $P = ev_a$.\footnote{ Since the terminal object $1$ is a separator in $\Cvx$, and the category $\I$ consist of a single object, it suffice to show the affine maps $\mbfI^A \rightarrow \mbfI^1$ correspond to a map $1 \rightarrow A$, hence a point of $A$.}

\begin{lemma} \label{wave} The commutativity condition holds if and only if  $P$ is a weakly averaging map.
\end{lemma}
\begin{proof}
If the affine map  $\mbfI^A \stackrel{P}{\longrightarrow} \mbfI$ is weakly averaging then $P(\overline{0})=0$ and $P(\overline{1})=1$, hence it satisfies the property 
\be \nonumber
P(s \cdot m) = P(s \cdot m + (1-s) \cdot \overline{0}) = P(m) +_{1-s} P(\overline{0}) = sP(m)
\ee
for all $m \in \mbfI^A$ and all $s \in \mbfI$, as well as  $P(\overline{t})=P(\overline{0}+_{t} \overline{1}) = t$.  Consequently, for $0 \le s+t \le 1$, it follows
\be \nonumber
P(s\cdot m + t) = P( \frac{s}{1-t} m +_{t} \overline{1}) = s P(m) + t.
\ee

Conversely,  if the  equation
\be \nonumber
P(s m + t) = s P(m) + t \quad \quad \textrm{ for all } \langle s, t \rangle \in \Cvx(\mbfI,\mbfI)
\ee
holds then $P$ is weakly averaging as the equation must hold true for $s=0$ and arbitrary $t \in \mbfI$.
\end{proof}

   Denote the subobject of $\Cvx(\mbfI^A, \mbfI)$ consisting of the weakly averaging functionals by 
   $\mbfI^{\mbfI^A}|_{wa}$.  Thus  $\mbfI^{\mbfI^A}|_{wa} \hookrightarrow \Cvx(\mbfI^A, \mbfI)$, and since every 
   evaluation map $ev_a$ is a weakly averaging functional, the double dualization map $\widetilde{\eta}_A$ defined in Corollary (\ref{ddMap}), decomposes as
\begin{equation}  \nonumber
\begin{tikzpicture}
  \node  (A) at  (0,0)  {$A$};
  \node   (IIA)   at   (4,0)  {$\mbfI^{\mbfI^{A}}|_{wa}$};
  \node   (IA)   at    (4, -2) {$\mbfI^{\mbfI^A}$};
  \draw[>->, above] (A) to node {$\hat{\eta}_A$} (IIA);
  \draw[>->, right] (IIA) to node {$\iota$} (IA);
  \draw[>->,below] (A) to node {$\widetilde{\eta}_A$} (IA);
  \end{tikzpicture}
  \end{equation} 
\noindent
Because the category $\Cvx$ is a regular category (it is an equational theory), each affine map has a regular-epi mono factorization. Since $\widetilde{\eta}_A$ is monic, its factorization 
\be \nonumber
\widetilde{\eta}_A = j \circ q
\ee
makes the map $q$ both  regular-epi and monic.  \emph{Editorial note: The following result is false. The truth of this fact follows from the development given in the updated version, referred to at the outset of the paper.  Namely, it is necessary to show $\two^{\two^{\Sigma A}}|_{wa} \cong \Sigma A$, and then proceed to ``extend the evaluation map $ev_a$, corresponding to $\two^{\Sigma  A} \stackrel{\chi_{P^{-1}}}{\longrightarrow} \two$, back to $\mbfI^{\mbfI^A}$ (restricted to affine maps)}. But a  regular-epi mono map is an isomorphism in any category, and so we conclude $q$ is an isomorphism.  Thus we have the commutative square

\begin{equation}  \nonumber
\begin{tikzpicture}
  \node  (A) at  (0,0)  {$A$};
  \node  (cA) at   (0, -2) {$Im(\widetilde{\eta}_A)$};
  \node   (IIA)   at   (4,0)  {$\mbfI^{\mbfI^{A}}|_{wa}$};
  \node   (IA)   at    (4, -2) {$\mbfI^{\mbfI^A}$};
  \draw[>->, above] (A) to node {$\hat{\eta}_A$} (IIA);
  \draw[>->, right] (IIA) to node {$\iota$} (IA);
%  \draw[>->,below] (A) to node {$\widetilde{\eta}_A$} (IA);
\draw[>->,below,dashed] (IIA) to node {$k$} (cA); 
 \draw[>->>,left] (A) to node {$q$} (cA);
 \draw[>->,below] (cA) to node {$j$} (IA);
  \end{tikzpicture}
  \end{equation} 
\noindent
and, since $Im(\widetilde{\eta}_A)$ is the largest subobject through which $\widetilde{\eta}_A$ factors, it follows there exist a unique map $k$ making the whole diagram commute.  Since $q$ is an isomorphism and $q = k \circ \hat{\eta}_A$, the map $r_A \stackrel{def}{=} q^{-1} \circ k$ is a retraction map for $\hat{\eta}_A$,
\be \nonumber
r_A \circ \hat{\eta}_A = id_A.
\ee

On the otherhand, the image $Im(\widetilde{\eta}_A)$ is the convex hull of all the evaluation maps $\{ev_a\}_{a \in A}$, which is the smallest convex space containing the image.  Since each $ev_a$ is weakly averaging,  it follows that $k$ itself is an isomorphism, and the coimage map is, up to isomorphism, just the double dualization map $\hat{\eta}_A$ having the codomain space as the weakly averaging functionals. 
% Thus, up to an isomorphism,  the double dualization  map, with the restricted codomain of $\mbfI^{\mbfI^A}|_{wa}$, is an isomorphism.

This result yields 

\begin{lemma}  \label{mainRes} In $\Cvx$, every weakly averaging affine map $\mbfI^A \stackrel{P}{\longrightarrow} \mbfI$ is an evaluation map.
\end{lemma}
\begin{proof}
The convex space of weakly averaging affine maps, $\mbfI^{\mbfI^A}|_{wa}$, is (up to isomorphism) the convex hull of the set of evaluation points, $\{ev_a\}_{a \in A}$.  Hence every weakly averaging affine map
$\mbfI^A \stackrel{P}{\longrightarrow} \mbfI$ is given by 
\be \nonumber
\begin{array}{lcl}
P &=& \sum_{i=1}^n \alpha_i ev_{a_i} \\
&=&  ev_{\sum_{i=1}^n \alpha_i a_i}
\end{array}
\ee
where the first equation follows since the convex hull of the set $\{ev_a\}_{a \in A}$ consist of all convex sums of such points.    The second equality follows from the fact the map $\hat{\eta}_A$ is affine.  All told,  $P$ is an evaluation map at some point $a \in A$.
\end{proof}

\begin{thm} The convex subspace $\I$ is codense in $\Cvx$. 
\end{thm}
\begin{proof} 
To prove the codensity of $\I$ in $\Cvx$ requires showing  the restricted dual Yoneda functor  $\Cvx^{op} \rightarrow \Cvx^{\I}$ is full and faithful.  Since the object $1$ is a separator,  it suffices to show that every natural transformation $\Cvx(\cdot, A) \stackrel{P}{\longrightarrow} \Cvx(\cdot, 1)$ arises from a point $1 \stackrel{a}{\longrightarrow} A$, so that the natural transformation is $\Cvx(\cdot, a)$, and that this point be unique, i.e., $\Cvx(\cdot, a_1) = \Cvx(\cdot,a_2)$ implies $a_1=a_2$.   The natural transformations given by $\Cvx(\cdot,a)$, evaluated at the only object in $\I$, are the evaluation maps.  Since the double dualization map $\hat{\eta}_A$ ($\widetilde{\eta}_A$), which is the unit of the double dualization monad into $\mbfI$, evaluated at the component $A$, is monic it follows that the restricted dual Yoneda functor is faithful.  The property of being full follows from Lemma \ref{mainRes} since, by Lemma \ref{wave}, every natural transformation $P$ is necessarily a weakly averaging affine functional.

\end{proof}

\section{Codensity of $\I$ in $\Cvx$ via Isbell Duality}  \label{altView}

Since $\Cvx$ is a SMCC which is complete and cocomplete, and  $\I$ is a dense subcategory of $\Cvx$, we can apply  the Isbell duality theorem,
\begin{equation}   \nonumber
 \begin{tikzpicture}[baseline=(current bounding box.center)]
 
  \node     (CI) at  (0,0)  {$\Cvx^{\I^{op}}$};
   \node   (IC)  at   (5,0)   {$(\Cvx^{\I})^{op}$};
  
  \draw[->,above] ([yshift=2pt] CI.east) to node {$\mathcal{O}$} ([yshift=2pt] IC.west);
   \draw[->,below] ([yshift=-2pt] IC.west) to node {$\mathbf{Spec}$} ([yshift=-2pt] CI.east);

  \node    (d) at   (8, 0)  {$\mathcal{O} \dashv \mathbf{Spec}$};
  
   \end{tikzpicture}
 \end{equation}
\noindent
where
\be \nonumber
\begin{array}{l}
\bigg( \mathbf{Spec}(A)\bigg)[\mbfI] = (\Cvx^{\I})^{op})(\Cvx(\mbfI, \cdot), A) \\
\quad \quad \quad \quad \quad \quad \textrm{ and } \quad \\
 \bigg( \mathcal{O}(\mathcal{F})\bigg)[\mbfI] = (\Cvx^{\I^{op}})(\mathcal{F}, \Cvx(\cdot, \mbfI))
 \end{array}.
\ee
The restricted Yoneda embedding 
\be \nonumber
\begin{array}{ccc}
\Cvx & \stackrel{Y|}{\longrightarrow} & \Cvx^{\I^{op}} \\
A &\mapsto& \hat{A}
\end{array}
\ee
which is full and faithful (since $\I$ is adequate), gives the representable functors \mbox{$\hat{A} = \Cvx(\cdot, A) \in_{ob} \Cvx^{\I^{op}}$}, and it follows that 
\be \nonumber
\begin{array}{lcl}
 \mathbf{Spec}(\mathcal{O}(\hat{A}))[\mbfI] &=& \mathbf{Spec}(\mathcal{O}(\Cvx(\cdot, A)))[\mbfI]  \\
 &\cong & \mathbf{Spec}( \Cvx(A, \cdot))[\mbfI] \\
 &=& (\Cvx^{\I})^{op}(\Cvx(I, \cdot), \Cvx(A, \cdot) ) \\
 &=& \Cvx(\cdot, A)[\mbfI] \\
 &=& \hat{A}[\mbfI]
 \end{array}
\ee
hence  the composite functor map
\be \nonumber
\hat{A} \mapsto  \mathbf{Spec}(\mathcal{O}(\hat{A}))
\ee
is an isomorphism for every convex space $A$.

 Since the restricted Yoneda embedding \mbox{$\Cvx \stackrel{Y|}{\hookrightarrow} \Cvx^{\I^{op}}$} is full and faithful, it follows  the composite functor \mbox{$\Cvx \stackrel{\mathcal{O} \circ Y|}{\longrightarrow} (\Cvx^{\I})^{op}$} is also full and faithful since $\mathcal{O}(\Cvx(\cdot,A)) = \Cvx(A, \cdot)$.  This implies that  the dual Yoneda embedding \mbox{$\Cvx^{op} \stackrel{\hat{Y}|}{\longrightarrow} \Cvx^{\I}$} is full and faithful - hence $\I$ is a right adequate subcategory of $\Cvx$.\footnote{This result is given in Isbell, using our notation, as 
\begin{thm} A proper left adequate subcategory is a proper right adequate subcategory if and only if all the representable functors 
\be \nonumber
Cvx(\cdot, A) \in \Cvx^{\I^{op}}
\ee
\noindent
are reflexive, for all $A \in \Cvx$.
\end{thm}
\begin{proof} \cite[Theorem 1.5]{Isbell}  (Given without proof.)
\end{proof}
The property of being proper is trivial since $\I$ consist of a single object.}

\section{Function Spaces as Positively Convex Spaces}  \label{functions}

Consider the convex space $\Cvx(A, \mbfI)$, with its convex structure determined pointwise by that of $\two$, so for all $\alpha \in (0,1)$, 
\be \nonumber
\begin{array}{lcl}
(\chi_{A_1} +_{\alpha} \chi_{A_2})(b) &\stackrel{def}{=}& \chi_{A_1}(b) +_{\alpha} \chi_{A_2}(b) \\
&=& \left\{ \begin{array}{ll} 1 & \textrm{iff }b \in A_1  \cap  A_2  \\ 0 & \textrm{otherwise} \end{array} \right.  
\end{array}
\ee
Thus, it is clear that for any finite convex sum of such elements in $\two^A$,  that we have
\be \nonumber
\displaystyle{ \sum_{i=1}^n} \alpha_i \chi_{A_i} = \chi_{ \cap_{i=1}^n A_i }  \quad \displaystyle{\sum_{i=1}^n} \alpha_i = 1,  \, \, \alpha_i \in (0,1) 
\ee
Moreover, since 
\be \nonumber
\bigcap_{i=1}^{\infty} A_i = \left\{\begin{array}{cl} A_k & \textrm{iff } \cap_{i=1}^{\infty} A_i \ne \emptyset \\
\emptyset & \textrm{otherwise }
\end{array} \right.
\ee
where, provided the intersection is nonempty, the element $A_k$ corresponds to one of the terms in set $\{A_i\}_{i=1}^{\infty}$.   Hence,  for every $b \in A$,
\be \nonumber
(\chi_{\cap_{i=1}^{\infty} A_i})(b)= \left\{ \begin{array}{ll} 1 & \textrm{iff }b \in A_i \textrm{ for all }i \\ 0 & \textrm{otherwise} \end{array} \right..  
\ee

Consequently, for every convex space $A$, the set $\Cvx(A, \two)$ has a superconvex  space structure given by  
\be \nonumber 
\displaystyle{ \sum_{i=1}^{\infty}} \alpha_i \chi_{ A_i }  \stackrel{def}{=}  \chi_{ \cap_{i=1}^{\infty} A_i }  \quad  \displaystyle{\lim_{n \rightarrow \infty}}  \left\{ \sum_{i=1}^n \alpha_i \right\}  = 1,  \, \, \alpha_i \in (0,1) 
\ee
with the nullary operation 
\be \nonumber
\begin{array}{lcl} 
\one & \rightarrow & \two^A \\
\star & \mapsto & \chi_{\emptyset}
\end{array}.
\ee
%Henceforth, the expression $\two^A$ will refer to the set $\Cvx(A, \two)$ with the superconvex space structure.

Similarly the convex space $\Cvx(A, \mbfI)$, with its convex structure determined pointwise by that of $\mbfI$, determines a superconvex space.  That is,  for all $b \in A$, infinite convex sums are defined by

\be \nonumber
\bigg(\displaystyle{ \sum_{i=1}^{\infty}} \alpha_i \chi_{A_i}\bigg)(b) \stackrel{def}{=}  \lim_{n \rightarrow \infty} \alpha_i \chi_{A_i}(b)  \quad \textrm{ where } \lim_{n \rightarrow \infty} \{\sum_{i=1}^n \alpha_i\} = 1,  \, \, \alpha_i \in (0,1) 
\ee
with the nullary operation 
\be \nonumber
\begin{array}{lcl} 
\one & \rightarrow & \mbfI^A \\
\star & \mapsto & \chi_{\emptyset}
\end{array}.
\ee
%The expression $\mbfI^A$ will refer to the set $\Cvx(A, \mbfI)$ with the superconvex space structure.

Using Theorem \ref{superconvex},  we will view these two superconvex spaces, $\two^A$ and $\mbfI^A$, as positively convex spaces.  Hence  we can replace the the condition of strict equality  $\lim_{n \rightarrow \infty} \{\sum_{i=1}^n \alpha_i\} = 1$ with  the condition of inequality $\lim_{n \rightarrow \infty} \{\sum_{i=1}^n \alpha_i\} \le 1$.  It is, of course, the zero element $\chi_{\emptyset}$ which allows us to add the additional term $(1- \lim_{n \rightarrow \infty} \sum_{i=1}^{n} \alpha_i)$ to obtain a superconvex space (given a positively convex space).

The property of being a superconvex (positively)  space works in tandem with the idea of measurable functions into $\mbfI$ being represented as limits of simple measurable functions, which themselves can be (re)written as convex sums.  These representations make explicit use of the zero elements of the function space $\mbfI^A$, as the proof of the following result shows. 

\begin{lemma}\label{wellDefined}  Every simple measurable function $\Sigma A \stackrel{m}{\longrightarrow} \mbfI$ can be written as a convex sum, $m = \sum_{i=1}^n \alpha_i \chi_{A_i}$ with $\sum_{i=1}^n \alpha_i =1$.
\end{lemma}
\begin{proof}  We  can assume the simple  measurable function  \mbox{$m = \sum_{i=1}^n \alpha_i \chi_{S_i}$} is written with 
 pairwise disjoint measurable sets $\{S_i \}_{i=1}^n$  and has increasing coefficients, $\alpha_1 \le \alpha_2, \ldots \le  \alpha_n$, and each $S_i$ is measurable. (Clearly, the $S_i$ will generally not be in the generating set for $\Sigma A$.)  Moreover, we can assume that $\cup_{i=1}^N S_i = A$.  (If not, add the complementary of the union and associate a $0$ coefficient with it.) This sum can be rewritten as the ``telescoping'' function
\be  \nonumber      %\label{telescope}
\begin{array}{lcl}
m &=& \alpha_1 \chi_{\cup_{i=1}^n S_i} +(\alpha_2 - \alpha_1) \chi_{\cup_{i=2}^n S_i}  + \ldots \\
&&  + (\alpha_j - \alpha_{j-1}) \chi_{\cup_{i=j}^n S_i} + \ldots  + (\alpha_n - \alpha_{n-1}) \chi_{S_n} + (1-\alpha_n) \chi_{\emptyset}
\end{array}
\ee
which satisfies the condition that the sum of the coefficients is one.
\end{proof}

\section{The map $\mbfI^{\epsilon_{\two}^A}$}
To \emph{motivate} the following construction, the the following observation may be useful. (It is by no means necessary; it is included to lend understanding to \emph{why} the map $\mbfI^{\epsilon_\two^A}$ may have relevance.)  In measure theory, it is well know that every measurable function from a measurable space into the real line can be represented as a limit of a sequence of simple functions.   Thus, in particular, taking the space $\two^A$, which is used to generate the $\sigma$-algebra on $\Sigma A$, we expect $\two^{\Sigma A}$ of all measurable characteristic functions, should yield all the measurable functions $\mbfI^{\Sigma A}$.  We have alread noted  that such function spaces can be viewed as superconvex spaces or positively convex spaces.

 Take the map $\epsilon_\two$, defined in equation (\ref{def2}), and exponentiate it by the (convex) space $A$ to obtain the standard covariant map, in $\Cvx$, given by
 
 \begin{equation}   \nonumber
 \begin{tikzpicture}[baseline=(current bounding box.center)]
 
   \node   (IA)  at   (-1,0)   {$\mbfI^{A}$};
  \node   (2A)    at   (5,0)   {$\two^{A}$};
  \draw[->,above] (IA) to node {$\epsilon_{\two}^{A}$} (2A);
  \node    (A) at   (-2,-1.9)   {$A$};
  \node    (I)  at    (-.3, -1.9)   {$\mbfI$};
  \draw[->,above] (A) to node {$m$} (I);
  \node   (A2) at (4, -1)  {$A$};
  \node    (I2)  at  (6, -1)  {$\mbfI$};
  \node    (2)   at   (6, -3)  {$\two$};
  \draw[->,above] (A2) to node {$m$} (I2);
  \draw[->,right] (I2) to node {$\epsilon_{\two}$} (2);
  \draw[ ->,left] (A2) to node {$\chi_{m^{-1}(1)}=\epsilon_{\two} \circ m$} (2);
       
  \node  (p) at    (1.8, -1.9)   {};
  \draw[|->,dashed] (I) to node {$$} (p);
 \end{tikzpicture}
 \end{equation}
\noindent
The composite map defines a Boolean subobject of $A$ given by 
$(\epsilon_{\two} \circ m)^{-1}(1)=m^{-1}(1)$, yielding  the property 
\be \label{epsm}
\epsilon_{\two} \circ m = \chi_{m^{-1}(1)} \quad \textrm{ for all }m \in \Cvx(A, \mbfI).
\ee
%and hence measurable under application of the functor $\Sigma$.

 Exponentiating this map $\epsilon_{\two}^{A}$ by the (convex) space $\mbfI$, we obtain the affine map 
 
 \begin{equation}   \nonumber
 \begin{tikzpicture}[baseline=(current bounding box.center)]
   \node   (wa)   at   (-1,0)   {$\mbfI^{\two^{A}}|_{wa}$};
 %  \node   (22A)  at   (1.1,0)   {$\mbfI^{\two^A}$};
  \node   (2A)    at   (4,0)   {$\mbfI^{\mbfI^{A}}$};
   \node   (I)      at    (-2, -1)  {$\two^{A}$};
   \node   (I1)   at     (0, -1)   {$\mbfI$};
   \draw[->,above] (wa) to node {$\mbfI^{\epsilon_{\two}^{A}}$} (2A);
   \draw[->, above] (I) to node {$P$} (I1);
   
   \node  (A2)   at   (3, -1)   {$\two^{A}$};
   \node   (2)     at   (5, -1)   {$\mbfI$};
   \node   (IA)    at   (3, -2.2)  {$\mbfI^{A}$};
   
   \draw[->, above] (A2) to node {$P$} (2);   
   \draw[->,left] (IA) to node {$\epsilon_{\two}^A$} (A2);
   \draw[->,right,below, thick] (IA) to node [xshift=5pt]{$P \circ \epsilon_{\two}^A$} (2);
   
   \node  (c1)  at   (8.3, 0)  {$\mbfI^{\two^{A}} = \Cvx(\Cvx(A, \two), \mbfI)$};
   \node   (c2)  at   (8.3, -.9)  {$\mbfI^{\mbfI^{A}} = \Cvx(\Cvx(A, \mbfI), \mbfI)$};
   
% \draw[->,above] (22A) to node {$\mbfI^{\epsilon_{\two}^A}$} (2A);
  \draw[|->,thick,dashed] (I1) to node {} (A2);
  
 \end{tikzpicture}
 \end{equation}
\noindent
where $\mbfI^{\two^{A}}|_{wa} \hookrightarrow \mbfI^{\two^{A}}$ is the subobject (subspace) consisting of the weakly averaging affine functionals\footnote{Recall, a weakly averaging functional $P$ sends a constant map to the value of the constant, $P(\overline{c})=c$.  The evaluation maps are clearly weakly averaging.},
\begin{equation}   \nonumber
 \begin{tikzpicture}[baseline=(current bounding box.center)]
 
   \node   (PA)  at   (0,0)   {$\two^{A}$};
  \node   (I2A)    at   (4,0)   {$\mbfI$};
%  \node   (0)     at    (0, -.8)   {$\ulcorner \chi_{\emptyset} \urcorner$};
%  \node   (f0)    at     (4, -.8)   {$0$};
%  \node   (1)    at     (0, -1.6)   {$\ulcorner \chi_{A} \urcorner$};
%  \node   (f1)   at     (4, -1.6)   {$1$};
   
   \draw[->, above] (PA) to node {$P$} (I2A);
 %  \draw[|->] (0) to node {} (f0);
 %  \draw[|->] (1) to node {} (f1); 
 \end{tikzpicture}
 \end{equation}
\noindent
sending the only two constant functions, $\chi_{\emptyset}$ and $\chi_{A}$, to their respective values,  $0$ and $1$.    
%Similarly, $\mbfI^{\mbfI^{A}}|_{wa}$ is the subobject of $\mbfI^{\mbfI^{A}}$ consisting of the weakly averaging affine functionals.

Using equation (\ref{epsm}), it follows that for every weakly averaging affine functional $P \in \mbfI^{\two^{A}}$ that
\be \label{Pproperty}
P(\ulcorner \epsilon_{\two} \circ m \urcorner) = P(\ulcorner \chi_{m^{-1}(1)} \urcorner).
\ee

\begin{lemma}
The image of the map $\mbfI^{\two^A}|_{wa} \stackrel{\mbfI^{\epsilon_{\two}^A}}{\longrightarrow} \mbfI^{\mbfI^A}$ lies in the subspace of weakly averaging affine functionals,  $\mbfI^{\mbfI^A}|_{wa}$.  
%Thus it contains all the weakly averaging affine functionals $\mbfI^A \stackrel{P}{\longrightarrow} \mbfI$. 

\end{lemma}
\begin{proof}
Using the trivial pair of Boolean subobjects, $\emptyset$ and $A$,   any constant function $A \stackrel{\overline{c}}{\longrightarrow} \mbfI$ can be written as the convex sum, 
\be \nonumber
\overline{c} = \overline{0} +_{c} \overline{1} = \chi_{\emptyset} +_{c} \chi_{A}. 
\ee
Every weakly averaging functional $P$ satisfies $P(\ulcorner \chi_{\emptyset} \urcorner)=0$ and \mbox{$P(\ulcorner \chi_A \urcorner)=1$}, and conversely if $P$ maps   $\chi_{\emptyset} \mapsto 0$ and $\chi_A \mapsto 1$ then it 
preserves all constant maps on $A$.  Under the mapping $\mbfI^{\epsilon_{\two}^A}$, $P \mapsto P\circ \epsilon_{\two}^A$, and consequently it suffices to show that $\mbfI^{\mbfI^{A}}|_{wa}$ contains all the functionals $P\circ \epsilon_{\two}^A$ such that 
\be \nonumber
(P \circ \epsilon_{\two}^A)(\ulcorner \chi_{\emptyset} \urcorner) = P(\ulcorner\epsilon_{\two} \circ  \chi_{\emptyset} \urcorner)=0
\ee
and
\be \nonumber
(P \circ \epsilon_{\two}^A)(\ulcorner \chi_{A} \urcorner) = P(\ulcorner\epsilon_{\two} \circ  \chi_{A} \urcorner)=1
\ee

Using equation (\ref{Pproperty}) and taking $m= \overline{0}=\chi_{\emptyset}$ we have, for every $a \in A$,
\be \nonumber
\begin{array}{lcl}
P(\ulcorner \epsilon_{\two} \circ \chi_{\emptyset} \urcorner) &=& P(\ulcorner \chi_{\chi_{\emptyset}^{-1}(1)} \urcorner) \\
&=& P(\ulcorner \chi_{\emptyset} \urcorner) \\
&=& 0 \\
&=& ev_a(\ulcorner \chi_{\emptyset} \urcorner)
\end{array}
\ee
Thus  for every weakly averaging functional $P \in \mbfI^{\two^A}$, the image of $P$ under $\mbfI^{\epsilon_{\two}^A}$ is, with respect to evaluation on the affine map $\chi_{\emptyset}$, equivalent to the evaluation map at any point $a \in A$.

Similarly, we have upon taking $m= \overline{1}=\chi_{A}$ the result that, for every $a \in A$, 
\be \nonumber
\begin{array}{lcl}
P(\ulcorner \epsilon_{\two} \circ \chi_{A} \urcorner) &=& P(\ulcorner \chi_{\chi_{A}^{-1}(1)} \urcorner) \\
&=& P(\ulcorner \chi_{A} \urcorner) \\
&=& 1 \\
&=& ev_a(\ulcorner \chi_{A} \urcorner)
\end{array}
\ee
Hence the image of every weakly averaging functional $P \in \mbfI^{\two^A}$ is, with respect to evaluation on the affine map $\chi_{A}$, equivalent to the evaluation map at any point $a \in A$.

Since all the evaluation maps lie in $\mbfI^{\mbfI^A}|_{wa}$, the result follows.

\end{proof}

\section{Proving the adjunction $\T \dashv \Sigma$}  \label{adjunction}

The unit of the proposed adjunction $\T \dashv \Sigma$ is necessarily the unit of the Giry monad, namely
\begin{equation}   \nonumber
 \begin{tikzpicture}[baseline=(current bounding box.center)]
 
  \node     (X) at  (0,0)  {$X$};
   \node   (TX)  at   (6,0)   {$\Sigma(\T(X))$};
  \node    (x)   at   (0, -.8)   {$x$};
  \node    (dx) at   (6, -.8)  {$\delta_x$};
  
  \draw[->,above] (X) to node {$\eta_X$} (TX);
   \draw[|->] (x) to node {} (dx);
 \end{tikzpicture}
 \end{equation}
where $\delta_x$ is the dirac (probability) measure at the point $x$.
 
The counit of the adjunction, at component $A$, 
\begin{equation}   \nonumber
 \begin{tikzpicture}[baseline=(current bounding box.center)]
 
  \node     (PA) at  (0,0)  {$\T(\Sigma(A))$};
   \node   (A)  at   (6,0)   {$A$};

  \draw[->,above] (PA) to node {$\epsilon_A$} (A);

 \end{tikzpicture}
 \end{equation}
 \noindent
is (as always) a universal arrow from the functor $\T$ to the object $A$, and  
the universal arrow $\epsilon_{\two}$ is as given in equation (\ref{def2}). Using the isomorphism between $\mbfI$ and $\T(\Sigma \two)$ we have
\begin{equation}   \nonumber
 \begin{tikzpicture}[baseline=(current bounding box.center)]
 
  \node     (P2) at  (0,0)  {$\T(\Sigma(\two))$};
   \node   (2)  at   (6,0)   {$\two$};
  \node   (PX)    at   (0,-2)   {$\T(X)$};
  \draw[->>,above] (P2) to node {\tiny{$\epsilon_{\two}(\alpha) \stackrel{def}{=} \left\{ \begin{array}{ll} 0 & \textrm{ for all }\alpha \in [0,1) \\ 1 & \textrm{ otherwise} \end{array} \right.$}} (2);    % \epsilon_{\two}
  \draw[->,left] (PX) to node {$\T(\hat{m})$} (P2);
  \draw[->,below] (PX) to node {$m$} (2);
  
  \node   (22) at  (-4, 0)  {$\Sigma(\two)$};
  \node  (X)  at   (-4, -2)  {$X$};
  \node   (SPX) at  (-2.5, -1)  {$\Sigma(\T(X))$};
  \draw[->,right,dashed] (X) to node [yshift=-2pt] {$\eta_X$} (SPX);
  \draw[->,right,dashed] (SPX) to node [yshift=2pt] {$\Sigma(m)$} (22);
  \draw[->,left] (X) to node {$\hat{m}$} (22);
  \node   (c)   at    (-4, -3)    {in $\M$};
  \node   (d)   at    (3.5, -3)   {in $\Cvx$};
  
%   \node   (g)    at    (7., -1)   {$\epsilon_{\two}(\alpha) \stackrel{def}{=} \left\{ \begin{array}{ll} 1 & \textrm{ for all }\alpha \in (0,1] \\ 0 & \textrm{ otherwise} \end{array} \right.$};

 \end{tikzpicture}
 \end{equation}
To show $\epsilon_{\two}$ is a universal arrow from $\T$ to $\two$, let $m$ be any affine map as shown in the diagram.  The adjunct to $m$ is 
\be \nonumber
\hat{m}(x) = \chi_{(\eta_X \circ \Sigma(m))^{-1}}(x) = \left\{ \begin{array}{ll} 1 & \textrm{iff }m(\delta_x)=1 \\  0 & \textrm{otherwise } \end{array} \right.,
\ee
specifying the subset of $X$ consisting of all those elements $x \in X$ such that the corresponding dirac measues $\delta_x$ get mapped to $1$ under the given affine map $m$.

The counit at the other components of $\Cvx$, $\epsilon_A$, can be determined using the affine map $\mbfI^{\epsilon_{\two}^A}$,  the result that the map  $A \stackrel{\widehat{\eta}_A}{\longrightarrow} \mbfI^{\mbfI^A}|_{wa}$ is an isomorphism, and the fact that the image $Im(\mbfI^{\epsilon_{\two}^A}) \subset \mbfI^{\mbfI^A}|_{wa}$.   
Thus, consider the $\Cvx$-diagram 
\begin{figure}[H]
 \begin{equation}   \nonumber
 \begin{tikzpicture}[baseline=(current bounding box.center)]
 
   \node   (A)  at   (.25,0)   {$A$};
  \node   (22A)    at   (4,0)   {$\mbfI^{\mbfI^A}|_{wa}$};
   \node   (I2A)      at    (4, 2)  {$\mbfI^{\two^{A}}|_{wa}$};
   \node   (PA)   at     (.25, 2)     {$\mbfI^{\two^{A}}|_{wa}$};
%   \node   (c)   at    (5,  -1)    {in $\Cvx$};
   
   \draw[->, above] (PA) to node {$id$} (I2A);
   \draw[->,left] (PA) to node {$\epsilon_A \stackrel{def}{=}\widehat{\eta}_A^{-1} \circ \mbfI^{\epsilon_{\two}^A}$} (A);
   \draw[>->>, above] ([yshift=2pt] A.east) to node {$\widehat{\eta}_A$} ([yshift=2pt] 22A.west);
   \draw[>->>,below] ([yshift=-2pt] 22A.west) to node {$\widehat{\eta}_A^{-1}$} ([yshift=-2pt]A.east);
   \draw[->,right] (I2A) to node {$\mbfI^{\epsilon_{\two}^{A}}$} (22A);

  \node   (a)  at   (6,0)   {$\epsilon_A(\ulcorner P \urcorner)$};
  \node   (Pres)    at   (10.6,2)   {$\ulcorner P \urcorner$};
   \node   (eP)      at    (10.6, 0)  {$\ulcorner P \circ \epsilon_{\two}^A  \urcorner$};
   \node   (P)   at     (6, 2)   {$\ulcorner P \urcorner$};
   \node   (a1)  at    (8.78, 0)   {$\ulcorner {ev}_{\epsilon_A(P)}\urcorner=$};
   
   \draw[|->, above,dashed] (P) to node {} (Pres);
   \draw[|->,left, dashed] (P) to node {} (a);
   \draw[|->, below] (a) to node {} (a1);
   \draw[|->,right] (Pres) to node {} (eP);
   
 \end{tikzpicture}
 \end{equation}
% \caption{Constructing the maps $\widetilde{\epsilon}_{A}$ as a pullback in $\Cvx$ using $\epsilon_{\two}$ and the evaluation maps. The space $\mbfI^{\two^A}|_{wa} \subset \Cvx( \M(\Sigma A, \two), \mbfI)$ consist of affine weakly averaging functionals. ($\two^A$ has the pointwise convex space structure.)  Similiarly,  the space $\mbfI^{\mbfI^{\Sigma A}}|_{wa} = \Cvx(\M(\Sigma A, \mbfI), \mbfI)$ consist of the weakly averaging affine functionals.}
\label{pb}
\end{figure}
\noindent
The fact the diagram is a pullback is trivial since $\widehat{\eta}_A$ is an isomorphism.  
\vspace{.1in}

The affine map  $\epsilon_A$  in turn determines  a map on  $\T(\Sigma A)$ using the affine map
sending a probability measure $\hat{P}$ (as traditionally viewed) to the weakly averaging affine map $P$,
 \begin{equation}   \nonumber
 \begin{tikzpicture}[baseline=(current bounding box.center)]
 
   \node   (I2A)      at    (5, 0)  {$\mbfI^{\two^{A}}|_{wa}$};
   \node   (PrA)  at    (0,0)    {$\T(\Sigma A)$};
   
   \draw[->, above] (PrA) to node {$\phi_A$} (I2A);

    \node   (P)   at    (0, -1)    {$\hat{P}$};
    \node    (2A) at    (4, -1)    {$\two^{A}$};
    \node    (I)    at    (6, -1)     {$\mbfI$};
    \draw[->, above] (2A) to node {$P$} (I);
    \draw[|->] (P) to node {} (2A);
    \node   (U)  at   (4, -1.8)   {$\chi_U$};
    \node   (PU) at  (6, -1.8)   {$\hat{P}(U)$};
    \draw[|->] (U) to node {} (PU);
    
 \end{tikzpicture}
 \end{equation}
To prove $\phi_A$ is an affine map we need only use the pointwise definition of the convex structure on the function space,
 \be \nonumber
 \begin{array}{lcl}
 \phi_A(\hat{P} +_{\alpha} \hat{Q})(\chi_U) &=& (P+_{\alpha} Q)(U) \\
 &=& P(U) +_{\alpha} Q(U) \\
 &=& (\phi_A(\hat{P}) +_{\alpha} \phi_A(\hat{Q}))(\chi_U)
 \end{array}.  
 \ee
 
 This affine map $\phi_A$ is an isomorphism with the inverse given by 
 \begin{equation}   \nonumber
 \begin{tikzpicture}[baseline=(current bounding box.center)]
 
   \node   (I2A)      at    (0, 0)  {$\mbfI^{\two^{A}}|_{wa}$};
   \node   (PrA)  at    (5,0)    {$\T(\Sigma A)$};
   
   \draw[->, above] (I2A) to node {$\phi_A^{-1}$} (PrA);

    \node   (P)   at    (5, -1)    {$\hat{P}$};
    \node    (2A) at    (-1, -1)    {$\two^{A}$};
    \node    (I)    at    (1, -1)     {$\mbfI$};
    \draw[->, above] (2A) to node {$P$} (I);
    \draw[|->] (I) to node {} (P);
%    \node   (U)  at   (-1, -1.8)   {$\chi_U$};
%    \node   (PU) at  (1, -1.8)   {$\hat{P}(U)$};
%    \draw[|->] (U) to node {} (PU);
    
 \end{tikzpicture}
 \end{equation}
where $\hat{P}(A_0) = P(\chi_{A_0})$ on the Boolean subobjects of $A$ which generate $\Sigma_A$.  Since the Boolean subobjects form,   by Lemma (\ref{piSystem}),  a $\pi$ system this completely defines the probability measure $\hat{P}$. 
Hence If $\hat{P}$ and $\hat{Q}$ are two probability measures on $\Sigma A$, which agree on the generating set, $\hat{P}(A_0)=\hat{Q}(A_0)$ for all Boolean subobjects $A_0$, then $\hat{P}=\hat{Q}$.\footnote{The set $D=\{U \in \Sigma_A \, | \, \hat{P}(U) = \hat{Q}(U) \}$ forms a Dynkin system. Thus $\Sigma_A$, which is generated by the Boolean subobjects, satisfies $\Sigma_A  \subset D$, since the Boolean subobjects are closed under finite intersection.}  Note that the value at any other measurable subset $U \in \Sigma A$ can be calculated as the limit of an infinite sum because $\mbfI$ is a positively convex space, i.e., $\mbfI$ can be viewed not only as a convex space, but as a positively convex space.  We subsequently drop reference to the isomorphism $\phi_A$ and view $\T(\Sigma A)$ as $\mbfI^{\two^A}|_{wa}$.

The maps $\{\phi_A\}_{A \in_{ob} \Cvx}$ form the components of a natural transformation $\T(\Sigma (\bullet)) \Rightarrow \mbfI^{\two^{\bullet}}$, which is an elementary verification noting 
 \begin{equation}   \nonumber
 \begin{tikzpicture}[baseline=(current bounding box.center)]
 
   \node   (I2A)      at    (0, 0)  {$\mbfI^{\two^{A}}|_{wa}$};
   \node   (PrA)  at    (3,0)    {$\T(\Sigma A)$};
   \node   (I2B)   at    (0, -2)  {$\mbfI^{\two^B}|_{wa}$};
   \node   (PrB)  at     (3, -2)  {$\T(\Sigma B)$};
   
   \draw[->, above] (I2A) to node {$\phi_A^{-1}$} (PrA);
   \draw[->,left] (I2A) to node {$\mbfI^{\two^m}$} (I2B);
   \draw[->,below] (I2B) to node {$\phi_B^{-1}$} (PrB);
   \draw[->,right] (PrA) to node {$\T\Sigma(m)$} (PrB);

   \node   (P)      at    (5, 0)  {$P$};
   \node   (PA)  at    (8.4,0)    {$\hat{P}$};
   \node   (P2m)   at    (5, -2)  {$P \circ \two^m$};
   \node   (PB)  at     (8.4, -2)  {$\hat{P}m^{-1}$};
   \node    (p2mB) at  (7,-2)  {$\widehat{P \circ \two^m}=$};
   
   \draw[|->, above] (P) to node {} (PA);
   \draw[|->,right] (PA) to node {} (PB);
   \draw[|->,left] (P) to node {} (P2m);
   \draw[|->,below] (P2m) to node {} (p2mB);    
 \end{tikzpicture}
 \end{equation}
\noindent
where the commutativity follows from
\be \nonumber
\widehat{ P \circ \two^m}(V) = P(\chi_{m^{-1}(V)}) = \hat{P}m^{-1}(V) \quad \forall V \in \sa_B.
\ee

Now we can finally show that the functors $\T$ and $\Sigma$ form an adjoint pair between $\M$ and $\Cvx$, and that the two adjunctions, $\mathcal{F}^{\G} \dashv \mathcal{U}^{G}$ and $\T \dashv \Sigma$, shown in Diagram (\ref{comparisonFunctor}),  yield the same monad on $\M$.
 
\begin{thm} \label{it}  The functor $\T$ is left adjoint to $\Sigma$, $\T \dashv \Sigma$, and is naturally isomorphic to the Giry monad,   
\be \nonumber
(\G, \eta, \mu) \cong (\Sigma \circ \T, \eta, \Sigma(\epsilon_{\T \_})).
\ee
\end{thm}
\begin{proof}
This functor $\mathbf{\Sigma}$ endows each convex space
of probability measures $\T(X)$ with the same $\sigma$-algebra as that associated with the Giry monad since the evaluation maps, which are used to define the $\sigma$-algebra for the Giry monad, pulls back subobjects (=intervals) of $\mbfI$ to subobjects of $\T(X)$,\footnote{Conversely, the image of every convex subspace of $\T(X)$ is a convex space under every affine map.}
%, and can be written as a countable union of intervals in $\mbfI$.}

\begin{equation}   \nonumber
 \begin{tikzpicture}[baseline=(current bounding box.center)]
 
   \node   (ab)  at   (5,1.5)   {$(a,b)$};
  \node   (P)    at   (0,1.5)   {$ev_U^{-1}((a,b))$};
  \node   (I) at  (5, 0)  {$\mbfI$};
  \node  (TX)  at   (0, 0)  {$\T(X)$};
  \node   (c)   at    (7, .0)    {$U \in \sa_X$};
  
 \draw[>->] (ab) to node {} (I);
 \draw[>->,dashed] (P) to node {} (TX);
 \node  (c)  at   (7,1.1)   {in $\Cvx$};

%  \draw[->,above] (I) to node {$\lbrack P,Q \rbrack$} (TX);
   \draw[->,below] (TX) to node {$ev_U$}(I);
   \draw[->,above, dashed] (P) to node {} (ab);

 \end{tikzpicture}
 \end{equation}
 
The two natural transformations,   $\T\circ \Sigma \stackrel{\epsilon}{\longrightarrow} id_{\Cvx}$ and $id_{\M} \stackrel{\eta}{\longrightarrow} \Sigma \circ \T$, where $\eta_X$ sends a point $x \in X$ to the dirac measure $\delta_x$, together yield the required  bijective correspondence.  Given a measurable function $f$ 
%\begin{figure}[H]
\begin{equation}   \nonumber
 \begin{tikzpicture}[baseline=(current bounding box.center)]
 
   \node   (X)  at   (-1,0)   {$X$};
  \node    (TX) at  (3,0)  {$\Sigma(\T(X))$};
   \node   (SA)  at  (3, -3)  {$\Sigma A$};
   \node   (c)    at   (2, -4)  {in $\M$};
   
   \node   (PX)  at   (6, 0)  {$\T(X)$};
   \node    (A)   at    (6, -3)  {$A$};
   \node    (PA) at    (7.5, -1.5)  {$\T(\Sigma A)$};
   \node    (c2)  at    (6.5, -4)   {in $\Cvx$};
   \node    (d2)  at    (10, -1.5)  {$\hat{f} \stackrel{def}{=} \epsilon_A \circ \T(f)$};
   
 \draw[->,above] (X) to node {$\eta_X$} (TX);
 \draw[->,right] (TX) to node {$\Sigma(\hat{f})$} (SA);
 \draw[->,below] (X) to node {$f$} (SA);
 
 \draw[->,left,dashed] (PX) to node {$\hat{f}$} (A);
 \draw[->, right] (PX) to node {$\T(f)$} (PA);
 \draw[->>,right] (PA) to node {$\epsilon_A$} (A);
 
 \end{tikzpicture}
 \end{equation}
 %\caption{The bijective correspondence between $\M(X, \Sigma(A))$ and $\Cvx(\T(X), A)$.}
% \label{BiCorr}
% \end{figure}
\noindent 
define  $\hat{f}=\epsilon_A \circ \T(f)$, which  yields
\be \nonumber
\Sigma(\epsilon_A \circ  \T(f))\circ \eta_X(x) = \Sigma(\epsilon_A \circ \delta_{f(x)}) = f(x)
\ee
proving  the  existence of an adjunct arrow to $f$.  The uniqueness then follows from the fact that if $g \in \Cvx(\T(X), A)$ also satisfies the required commutativity condition of the diagram on the left,   $\Sigma g \circ \eta_X = f$, which says that for every  $x \in X$ that
\be \nonumber
g( \delta_x) = f(x) = \epsilon_A(\delta_{f(x)}) =  (\epsilon_{A} \circ \T(f))(\delta_x) = \hat{f}(\delta_x).
\ee
We can now use the fact that $\epsilon_{\T(X)} \circ \T(\eta_X) = id_{\T(X)}$  to conclude that for an arbitrary probability measure $P \in \T(X)$ that
 $g(P) = \hat{f}(P)$ follows using $g(P) = g(\epsilon_{\T(X)}(\delta_{P}))$ and naturality, 
 
 \begin{equation}   \nonumber
 \begin{tikzpicture}[baseline=(current bounding box.center)]
 
   \node   (PX)  at   (0,0)   {$\T(X)$};
  \node    (PSPX) at  (4,0)  {$\T(\Sigma(\T(X)))$};
  \node    (PX0)  at    (8,0)   {$\T(X)$};
  
   \node   (PA)  at  (4, -2)  {$\T(\Sigma A)$};
   \node   (A)  at   (8, -2)  {$A$};
   
 \draw[->,above] (PX) to node {$\T(\eta_X)$} (PSPX);
 \draw[->>,above] (PSPX) to node {$\epsilon_{\T(X)}$} (PX0);
 \draw[->,below] (PX) to node [xshift=-4pt] {$\T(f)$} (PA);
 \draw[->,right] (PSPX) to node {$\T(\Sigma g)$} (PA);
 \draw[->>, below] (PA) to node {$\epsilon_A$} (A);
 \draw[->,right] (PX0) to node {$g$} (A);

   \node   (P)  at   (0,-3)   {$P$};
  \node    (deltaP) at  (4,-3)  {$\delta_P$};
  \node    (P2)  at    (8,-3)   {$P$};
  
   \node   (Pf)  at  (4, -5)  {$Pf^{-1} \sim \delta_{g(P)}$};
   \node   (gP)  at   (8, -5)  {$\hat{f}(P) = g(P)$};
   
 \draw[|->] (P) to node {} (deltaP);
 \draw[|->] (deltaP) to node {} (P2);
 \draw[|->] (deltaP) to node {} (Pf);
 \draw[|->] (Pf) to node {} (gP);
 \draw[|->] (P2) to node {} (gP);
 \draw[|->] (P) to node {} (Pf);

 \end{tikzpicture}
 \end{equation}
\noindent 
where the bottom path,  $\hat{f} = \epsilon_A \circ \T(f)$ yields $\hat{f}(P)$,  while the east-south path  gives $g(P) = g(\epsilon_{\T(X)} \circ \T(\eta_X))$.

 The unit of $\T \dashv \Sigma$ is  $\eta$ (the same as the Giry monad), and the multiplication determined by the adjunction $\T \dashv \Sigma$ is given by
\be \nonumber
\tilde{\mu}_X = \Sigma( \epsilon_{\T(X)})
\ee
where the functor $\Sigma$ just makes the affine map $\epsilon_{\T(X)}$ a measurable function.
We must show that this $\tilde{\mu}$ coincides with the multiplication $\mu$ of the Giry monad which is defined componentwise by 
\be \nonumber
\mu_X(P)[U] = \int_{q \in \G(X)} ev_U(q) \, dP(q).
\ee
This follows by the naturality of $\epsilon$ and the fact $\T(\chi_U) = ev_U$ for all $U \in \sa_X$,
  \begin{equation}   \nonumber
 \begin{tikzpicture}[baseline=(current bounding box.center)]
         \node   (STSTX)  at   (0,0)   {$\Sigma(\T(\Sigma(\T(X))))$};
         \node  (STX)     at   (4,0)   {$\Sigma(\T(X))$};
         \node  (STST2)     at  (0, -2)  {$\Sigma(\T\Sigma(\T(\two)))$};
         \node   (ST2)    at   (4, -2)    {$\Sigma(\T(\two))$};
   
         \draw[->,left] (STSTX) to node {$\Sigma(\T(ev_U))$} (STST2);
         \draw[->,right] (STX) to node  {$ev_U$} (ST2);
         \draw[->, above]  (STSTX) to node {$\Sigma(\epsilon_{\T(X)})$} (STX);
         \draw[->, below] (STST2) to node [xshift=3pt]{$\Sigma(\epsilon_{\T(\two)})$} (ST2);
%         \draw[->, above] (A) to node {$\chi_U$} (2);
%         \draw[->,right] (I1) to node {$\langle s, t \rangle$} (I2);
         \node   (P)  at   (6,0)   {$P$};
         \node  (Q)     at   (11.5,0)   {$Q$};
         \node  (Pev)     at  (6, -2)  {$Pev_U^{-1}$};
         \node   (QU)    at   (9.5, -2)    {$\Sigma(\epsilon_{\T(\two)})(Pev_U^{-1}) = Q(U)$};
         \node    (p)       at    (11.5, -1.7)  {};
         \node    (c)    at    (8.4, -1)  {where $P \sim \delta_Q$};
         \draw[|->] (P) to node {} (Pev);
         \draw[|->] (Q) to node  {} (p);
         \draw[|->]  (P) to node {} (Q);
         \draw[|->] (Pev) to node [xshift=0pt]{} (QU);

	 \end{tikzpicture}
 \end{equation}
The east-south path gives $\tilde{\mu}_{\T(X)}\Sigma(\epsilon_{\T(X)}(P))[U] = Q(U)$, while the south-east path gives the multiplication $\mu_{X}$ of the Giry monad, 
\be \nonumber
\Sigma(\epsilon_{\T(\two)})(Pev_U^{-1}) = \int_{q \in \G(X)} q(U) \, dP = \mu_{X}(P)[U]. 
\ee
where we have used the fact $\mu_{\two} = \Sigma(\epsilon_{\T(\two)})$.

%The equality in the lower right hand corner of the diagram, which follows by the naturality, is just the statement that $\delta_{Q(U)} \sim Pev_U^{-1}$.
\end{proof}

\section{The equivalence of Giry algebras with convex spaces}   \label{equiv}

To show that the category $\M^{\G}$ is equivalent to $\Cvx$, we need to show that the adjoint pair\footnote{Since $\Cvx$ is cocomplete, the left adjoint $\hat{\Phi}$  to the comparison functor exist. (This left adjoint construction always exist whenever the category in question, here $\Cvx$, has coequalizers.)  The question of equivalence amounts to showing that the two composites of those functors are naturally isomorphic to the identity on the respective categories.

The fact that $\Sigma$ reflects isomorphisms is trivial since, for every affine map $m$,  $\Sigma m$ is just an  affine measurable function.  This implies that the counit of the adjunction $\hat{\Phi} \dashv \Phi$ is an isomorphism, and it only remains to show the unit of the adjunction is an isomorphism.}
\begin{equation}   \nonumber
 \begin{tikzpicture}[baseline=(current bounding box.center)]
         \node   (MG)  at   (0, 0)  {$\M^{\G}$};
         \node  (C) at  (5,0)   {$\Cvx$};
         \node  (c)   at    (8, 0)   {$\hat{\Phi} \dashv \Phi$};
         
        \draw[->,below] ([yshift=-2pt] C.west) to node {$\Phi$} ([yshift=-2pt] MG.east);
        \draw[->,above] ([yshift=2pt] MG.east) to node {$\hat{\Phi}$} ([yshift=2pt] C.west);

	 \end{tikzpicture}
 \end{equation}
 \noindent
have the unit and counit of the adjunction  naturally isomorphic to the identity functors on $\Cvx$ and $\M^{\G}$.  The functor $\hat{\Phi}$, applied to a Giry algebra $\G(X) \stackrel{h}{\longrightarrow} X$,  is the coequalizer (object) of the parallel pair\footnote{Recall, by Theorem \ref{it}, $\mu_X =\Sigma(\epsilon_{\T(X)})$.} 

\begin{equation}   \nonumber
 \begin{tikzpicture}[baseline=(current bounding box.center)]
         \node  (GX)  at (-0,0)    {$\T(\G(X))$};
         \node  (X)    at  (4,0)    {$\T(X)$};
         \node   (CC) at  (8,0)  {$CoEq$};
%        \node   (E)  at   (4,-2)  {$Ker(q_1,q_2)$};
         \draw[->>,above] (X) to node {$q$} (CC);

        \draw[->,above] ([yshift=2pt] GX.east) to node {$\epsilon_{\T(X)}$} ([yshift=2pt] X.west);
        \draw[->,below]  ([yshift=-2pt] GX.east) to node {$\T(h)$} ([yshift=-2pt] X.west);

	 \end{tikzpicture}
 \end{equation}
%where the map $\mu_X$, or equivalently using Theorem \ref{it}, $\Sigma(\epsilon_{\T(X)})$, is an affine map (and hence viewed in $\Cvx$).
 The convex space $Coeq$ is the $\Cvx$ object corresponding to the Giry algebra $h$.  For the equivalence to hold, we must show that the map of $\G$-algebras
 
 \begin{equation}   \nonumber
 \begin{tikzpicture}[baseline=(current bounding box.center)]
         \node  (GX)  at (0,0)    {$\G(X)$};
         \node  (X)    at  (0,-2)    {$X$};
         
         \node   (CoEq) at  (4,0)  {$\Sigma(\T(\Sigma CoEq))$};
         \node   (Co)      at   (4, -2) {$\Sigma CoEq$};
         \node     (c)     at     (7, -1)  {in $\M$};
%        \node   (E)  at   (4,-2)  {$Ker(q_1,q_2)$};
         \draw[->,left] (GX) to node {$h$} (X);
         \draw[->, right] (CoEq) to node {$\Sigma(\epsilon_{Coeq})$} (Co);
         \draw[->, above] (GX) to node {$\G(\theta)$} (CoEq);
         \draw[->,below] (X) to node {$\theta$} (Co);
        
	 \end{tikzpicture}
 \end{equation}
 is an isomorphism.   Towards this end, let $(Ker(q), m_1, m_2)$ denote the kernel pair of the coequalizer $q$.  Since $q \circ \epsilon_{\T(X)} = q \circ \T(h)$ there exist a unique map $\psi$ such that the $\Cvx$-diagram
 \begin{equation}   \nonumber
 \begin{tikzpicture}[baseline=(current bounding box.center)]
         \node   (GGX) at  (0,2.5)   {$\T(\G(X))$};
         \node  (GX)    at  (0,0)    {$\T(X)$};
         \draw[->,right] ([xshift=2pt] GGX.south) to node [xshift=2pt] {$\T(h)$} ([xshift=2pt] GX.north);
         \draw[->,left] ([xshift=-2pt] GGX.south) to node {$\epsilon_{\T(X)}$} ([xshift=-2pt] GX.north);

         \node   (E)  at   (-3,0)  {$Ker(q)$};
%         \node   (Ep) at  (-5.1,-.1)  {};
%         \node   (CCp) at (-5.1,-3.9) {};
         \node   (CC) at  (0,-1.9)  {$CoEq$};
%        \draw[->,above] ([xshift=-2pt] GX.south) to node [xshift=-5pt]{$\Sigma(q_1)$} ([xshift=-2pt] CC.north);
        \draw[->>,right]  ([xshift=0pt] GX.south) to node [xshift=5pt]{$q$} ([xshift=0pt] CC.north);
        \draw[->,above] ([yshift=2pt] E.east) to node {$m_1$} ([yshift=2pt] GX.west);
        \draw[->, below] ([yshift=-2pt] E.east) to node {$m_2$} ([yshift=-2pt] GX.west);

        \draw[->,left] (GGX) to node [yshift=3pt]{$\psi$} (E);

         \node   (rGGX) at  (7,2.25)   {$\G^2(X)$};
         \node  (rGX)    at  (7,0)    {$\G(X)$};
         \draw[->>,right] ([xshift=2pt] rGGX.south) to node [xshift=2pt] {$\G(h)$} ([xshift=2pt] rGX.north);
         \draw[->>,left] ([xshift=-2pt] rGGX.south) to node {$\mu_X$} ([xshift=-2pt] rGX.north);
         \node  (rX)      at   (7,-2)    {$X$};
         \node   (rCC) at  (3,-2)  {$\Sigma CoEq$};    \node (prCC) at  (3.2, -1.8)  {};
         \node   (rE)  at   (3,0)  {$\Sigma Ker(q)$};
         \node   (rEp) at  (3,-.1)  {};
         \node   (rCCp) at (3,-1.9) {};
%         \node    (rc)    at    (9,-3)    {in $\M$};
        \draw[->>,right] (rGX) to node {$h$} (rX);
%        \draw[->,above] ([xshift=-2pt] GX.south) to node [xshift=-5pt]{$\Sigma(q_1)$} ([xshift=-2pt] CC.north);
        \draw[->,below]  (rGX.south) to node [xshift=5pt]{$\Sigma q$} (prCC.north);
        \draw[->,above] ([yshift=2pt] rE.east) to node {$\Sigma m_1$} ([yshift=2pt] rGX.west);
        \draw[->, below] ([yshift=-2pt] rE.east) to node {$\Sigma m_2$} ([yshift=-2pt] rGX.west);
        \draw[->,left] (rEp) to node [xshift=0pt] {\tiny{$\Sigma(q \circ m_1)$}} (rCCp);
        \draw[->,below,dashed] (rX) to node {$\theta$} (rCC);
        \draw[->,left] (rGGX) to node [xshift=-4pt]{$\Sigma \psi$} (rE);

	 \end{tikzpicture}
 \end{equation}
 \noindent
 commutes.  Now apply the functor $\Sigma$ to this diagram to obtain the commutative $\M$-diagram on the right hand side of the above diagram.
 
%   Applying the functor $\Sigma$ to the above diagram, use $\Sigma \circ \T = \G$ (as functors on $\M$) and, for visualization purposes, rotate the diagram to obtain

Defining
 \be \nonumber
 \theta = \Sigma q \circ \eta_X 
 \ee
 in this diagram, we obtain the (equivalent but redrawn)  commutative diagram 
   \begin{equation}   \nonumber
 \begin{tikzpicture}[baseline=(current bounding box.center)]

         \node  (G2X)    at  (-5,4)    {$\G^2(X)$};
         \node   (GX)   at    (-5,0)  {$\G(X)$};
         \node   (GX2) at    (0,4)   {$\G(X)$};
         \node    (ker)  at    (-2.5,2)  {$\Sigma ker(q)$};
         \node    (X)     at      (2.,-2.)  {$X$};
         \node    (CC)  at     (0,-0)  {$\Sigma CoEq$};
         \node    (Xp)   at      (2.3,-2.4)  {};
         \node    (CCp) at    (-.1, -.1)    {};
 %        \node     (c)    at       (5,1)  {$\theta \stackrel{def}{=} \Sigma(q_1) \circ \eta_X = \Sigma(q_2) \circ \eta_X$};
        \draw[->,above] (G2X) to node {$\G(h)$} (GX2); 
        \draw[->,out=-30, in=90,looseness=.5, right] (GX2) to node {$h$} (X);
        \draw[->,left] (G2X) to node {$\mu_X$} (GX);
        \draw[->,out=-30,in=180,looseness=.5,below] (GX) to node {$h$} (X);    
        \draw[->, above] (X) to node {$\theta$} (CC);
        \draw[->,below,dashed] (CCp) to node [xshift=-5pt]{$\theta^{-1}$} (Xp);
        \draw[->,below] (GX) to node {$\Sigma q$} (CC);
        \draw[->,right] (GX2) to node [xshift=0pt] {$\Sigma q$} (CC);
        \draw[->,right] (G2X) to node {$\Sigma \psi$} (ker);
        \draw[->,right] (ker) to node [yshift=-3pt]{$\Sigma m_1$} (GX2);
        \draw[->,right] (ker) to node [xshift=3pt]{$\Sigma m_2$} (GX);
\end{tikzpicture}
 \end{equation}
 \noindent 
where the outer paths commute since the $\G$-algebra $h$ satisfies the condition that $h \circ \G(h) = h \circ \mu_X$.  
%  The pullback diagram of  $\{\Sigma q, \Sigma q\}$
 %then shows that $\mu_X = \Sigma(m \circ \psi)$ (as well as $\G(h) = \Sigma(m \circ \psi)$ - which we already knew).
 
 Since $(\Sigma(ker(q)), \Sigma m_1, \Sigma m_2)$  is a pullback and $h \circ \Sigma m_1 = h \circ \Sigma m_2$, there exist a 
 unique map $\Sigma CoEq \longrightarrow X$ which (necessarily) is the inverse of $\theta$, and hence we obtain the isomorphism of measurable spaces
\be \nonumber
X \cong \Sigma CoEq.
\ee
This proves that the unit of the adjunction $\hat{\Phi} \dashv \Phi$ is naturally isomorphic to the identity functor $id_{\M}$.
 
Conversely, given a convex space $C$, applying the functor $\Sigma$ to the counit of the adjunction at $C$ gives the $\G$-algebra\footnote{This proof this is a $\G$-algebra is a straightforward verification.}
\begin{equation}   \nonumber
 \begin{tikzpicture}[baseline=(current bounding box.center)]
         \node  (PC)  at (1,0)    {$\Sigma(\T(\Sigma C))$};
         \node  (C) at  (5,0)   {$\Sigma C$};
         
        \draw[->,above] (PC) to node {$\Sigma \epsilon_C$} (C);

	 \end{tikzpicture}
 \end{equation}
\noindent
This is precisely the process of applying the comparison functor $\Phi$ to the convex space $C$.
Now if we apply the preceding process (apply the functor $\hat{\Phi})$ to this $\G$-algebra, we construct the coequalizer of the parallel pair
\begin{equation}   \nonumber
 \begin{tikzpicture}[baseline=(current bounding box.center)]
         \node  (PC)  at (1.4,0)    {$\T(\Sigma(\T(\Sigma C)))$};
         \node  (C) at  (6,0)   {$\T(\Sigma C)$};
         \node   (A) at  (9, 0)  {$C$};
         
        \draw[->,above] ([yshift=2pt] PC.east) to node {$\epsilon_{\T(\Sigma C)}$} ([yshift=2pt] C.west);
        \draw[->,below] ([yshift=-2pt] PC.east) to node {$\T (\Sigma \epsilon_C)$} ([yshift=-2pt] C.west);
        \draw[->,above,dashed] (C) to node {$\epsilon_C$} (A);
	 \end{tikzpicture}
 \end{equation}
\noindent
which is, up to isomorphism, just the  convex space $C$ and counit of the adjunction $\T\dashv \Sigma$.

\appendix
\appendixpage
%The following material, which is useful for understanding various aspects of the duality, is not required to prove the main results.  
\section{Some aspects of the SMCC structure of $\M$}

The tensor product construction yielding  the monoidal structure $(\M, \otimes, 1)$ can, depending upon the $\sigma$-algebra structure of the space, have properties quite different from that of the product monoidal structure as the following result and example illustrate.

\begin{lemma} If $X \stackrel{f}{\longrightarrow} Y$ is not a constant function then the graph $\Gamma_f$ is not necessarily measurable.
\end{lemma}

This observation and following example that the graph function of a measurable function need not be measurable using the tensor $\sigma$-algebra are due to Hongseok Yang.

 \begin{example}
Consider $X = Y = \mathbb{R}$ and let the $\sigma$-algebra on $\mathbb{R}$ consist of those subsets $A \subset \mathbb{R}$ which are either countable or cocountable (the complement $A^c$ is countable)  Then the diagonal map $\Delta_{\mathbb{R}}: \mathbb{R} \rightarrow \mathbb{R} \otimes \mathbb{R}$ is not measurable.  To see this,  choose a subset $R_0\subset \mathbb{R}$  such that both $R_0$ and $R_0^c$ are uncountable. Let
\mbox{$A_0 = \{ (r,r) \in \mathbb{R} \times  \mathbb{R} \, | \,  r \in R_0 \}$}. Then $A_0$ is a measurable set of $\mathbb{R} \otimes \mathbb{R}$ because for any 
 $x \in \mathbb{R}$ it follows that the preimage of the constant graph function
 \be \nonumber
\Gamma_x^{-1}(A_0) = \left\{ \begin{array}{cl} \{x\} & \textrm{ iff } x \in R_0 \\ \emptyset & \textrm{ otherwise } \end{array} \right.
\ee
which is at most countable and hence a measurable set.
However $\Delta_{\mathbb{R}}^{-1}(A_0) = R_0 \not \in \sa_{\mathbb{R}}$, i.e., is not a measurable set.  Consequently the graph of the identity map on $\mathbb{R}$, $\Delta_{\mathbb{R}} = \Gamma_{id_{\mathbb{R}}}$, is not measurable.  
\end{example}

On the otherhand, there are also aspects associated with the tensor product monoidal structure which the product monoidal structure does not enjoy (generally because the evaluation maps need not be measurable using the product monoidal structure).

Define the map $\ge$
\be \nonumber
 \begin{tikzpicture}[baseline=(current bounding box.center)]
      \node      (II)    at   (0,0)   {$I \otimes I$};
     \node     (I)    at   (3,0)    {$I$};
     \node      (u)   at   (0,-.8)  {$(u,v)$};
     \node      (duv) at  (3,-.8)  {$\left\{ \begin{array}{ll} 1 & \textrm{ iff }v \le u \\ 0 & \textrm{ otherwise } \end{array} \right.$};
     
     \draw[->,above](II) to node {$\ge$} (I);
     \draw[|->] (u) to node {} (duv);
 \end{tikzpicture}.
\ee
which has the adjunct  denoted $\E$
\be \nonumber
 \begin{tikzpicture}[baseline=(current bounding box.center)]
      \node      (III)    at   (0,0)   {$I^I \otimes I$};
     \node     (I)    at   (3,0)    {$I$};
     \node     (II)     at   (0,-2)   {$I \otimes I$};
     \node     (I2)         at    (-3,0)   {$I^I$};
     \node     (I3)    at   (-3,-2)   {$I$};
     
     \node      (duv)   at   (5,0)  {$(\chi_{[0,u]},v)$};
     \node     (uv)  at   (5,-2) {$(u,v)$};
     \node      (one) at  (8,0)  {$\chi_{[0,u]}(v)$};
    
     \draw[->,above](III) to node {$ev_I$} (I);
     \draw[->,below] (II) to node [xshift=3pt]{$\ge$} (I);
     \draw[->,left] (II)  to node {$\E \otimes id_I$} (III);
     \draw[->,left] (I3) to node {$\E$} (I2);
     
     \draw[|->] (uv) to node {} (duv);
     \draw[|->] (uv) to node {} (one);
     \draw[|->] (duv) to node {} (one);

 \end{tikzpicture}
\ee
where $\E(u) = \chi_{[0,u]}$.  As the characteristic functions are measurable on $I$ this is a measurable function.
This map $\E$ possesses many of the familiar properties of an effect algebra.
 
 This map $\E$ is a section of the Lebesque measure on the unit interval. 
  
 \begin{lemma} \label{id} The Lebesque measure $\leb \in \T(I)$ is the unique probability measure on $I$ making the diagram 
 \be \nonumber
 \begin{tikzpicture}[baseline=(current bounding box.center)]
      \node      (I)    at   (0,0)   {$I$};
     \node     (I2)    at   (4,0)    {$I$};
%     \node      (I3)    at  (6,0)   {$I$};
     \node      (II)   at   (2,1.5)  {$I^I$};
     
     \draw[->,below](I) to node {$id_I$} (I2);
     \draw[->,above] (I) to node [xshift=-5pt]{$\E$} (II);
     \draw[->,above] (II) to node [xshift=2pt]{$\leb$} (I2);
 \end{tikzpicture}
\ee
 commute.
 \end{lemma}
 \begin{proof} The set $\{[0,r]\}_{r \in I}$ forms a $\pi$-system and generates the Borel $\sigma$-algebra on $I$. \end{proof}

Composition on the left of the composite $\leb \circ \E$  by a probability measure $P \in \T(X)$ gives the trivial observation that

\be \nonumber
 \begin{tikzpicture}[baseline=(current bounding box.center)]
 
        \node      (dI)    at   (-2,0)   {$I^I$};
     \node     (dII)    at   (-2,-1.5)    {$I$};
     \node     (dI2)     at   (-2,-3)   {$I^X$};
     
     \draw[->,left](dII) to node {$\E$} (dI);
     \draw[->,left] (dI2) to node {$P$} (dII);

      \node      (II)    at   (0,0)   {$I^I$};
     \node     (I)    at   (2.5,0)    {$I$};
     \node     (III)     at   (0,-1.5)   {$I$};
     \node     (II2)    at  (0,-3)   {$I^X$};
     \node    (c)   at   (5,-1.5)   {$\forall f \in I^X \quad  \leb( \E_{P(f)}) = P(f)$};
     
     \draw[->,above](II) to node {$\leb$} (I);
     \draw[->,left] (III) to node {$\E$} (II);
     \draw[->,left] (II2) to node {$P$} (III);
     \draw[->,right] (II2) to node {$P$} (I);
     \draw[->,left] (III) to node [yshift=4pt]{$id_I$} (I);
 
 \end{tikzpicture}
\ee
%\end{proof}
\noindent
which shows that every probability measure on any measurable space $X$ can be viewed in terms of the Lebesque probability measure on $\mbfI$ by pushing a measurable function $X \stackrel{f}{\longrightarrow} \mbfI$ forward to the characteristic function $\chi_{[0, P(f)]}$.

\newpage

\bibliographystyle{plain}

 \end{document}